\documentclass[11pt]{amsart}
\usepackage{amssymb}
\usepackage{graphicx}
\usepackage{xcolor} % A package to add color.
\usepackage{tensor}
\usepackage{fullpage} % Sets all margins to 1 in.  
\usepackage{amsmath}
\usepackage{amsthm}
\usepackage{verbatim}
\usepackage{hyperref}
\usepackage{array} 
\usepackage{enumitem}
\setlist[enumerate]{leftmargin=1.2em}
\setlist[itemize]{leftmargin=1.2em}
\usepackage{seqsplit,mathtools}
\mathtoolsset{showonlyrefs}

\definecolor{green}{rgb}{0,0.5,0} % Redefines the color green.
\newtheorem{theorem}{Theorem}[section]
\newtheorem{thm}{Theorem}[section]

\newtheorem{lem}[theorem]{Lemma}

\newtheorem{proposition}[theorem]{Proposition}
\newtheorem{prop}[theorem]{Proposition}

\theoremstyle{definition}

\theoremstyle{remark}

\newtheorem{rem}[theorem]{Remark}

\numberwithin{equation}{section}
\numberwithin{equation}{section}
\newcommand{\nrm}[1]{\Vert#1\Vert}

\newcommand{\br}[1]{\overline{#1}}

\newcommand{\nnrm}[1]{{\vert\kern-0.25ex\vert\kern-0.25ex\vert #1 
		\vert\kern-0.25ex\vert\kern-0.25ex\vert}}

\newcommand{\supp}{{\mathrm{supp}}\,}

\newcommand{\lap}{\Delta}

\newcommand{\rd}{\partial}
\newcommand{\nb}{\nabla}

\newcommand{\ift}{\infty}

\newcommand{\alp}{\alpha}
\newcommand{\bt}{\beta}
\newcommand{\gmm}{\gamma}

\newcommand{\lmb}{\lambda}

\newcommand{\tht}{\theta}

\newcommand{\omg}{\omega}
\newcommand{\Omg}{\Omega}

\newcommand{\bbR}{\mathbb R}

\newcommand{\calF}{\mathcal F}

\begin{document}
	
	\bibliographystyle{plain}
	\title{On the optimal rate of vortex stretching for axisymmetric Euler flows without swirl}

	\renewcommand{\thefootnote}{\fnsymbol{footnote}}
	\footnotetext{\emph{2020 AMS Mathematics Subject Classification:} 76B47, 35Q35}
	\footnotetext{\emph{Key words: Vortex stretching; axisymmetric flows without swirl; vorticity; Biot--Savart law} }
	\renewcommand{\thefootnote}{\arabic{footnote}}

	\author{Deokwoo Lim}
	\address{The Research Institute of Basic Sciences, Seoul National University, 1 Gwanak-ro, Gwanak-gu, Seoul 08826, Republic of Korea.}
	\email{dwlim95@snu.ac.kr}

	\author{In-Jee Jeong}
	\address{Department of Mathematical Sciences and RIM, Seoul National University, 1 Gwanak-ro, Gwanak-gu, Seoul 08826, and School of Mathematics, Korea Institute for Advanced Study, Republic of Korea.}
	\email{injee$ \_ $j@snu.ac.kr}

	\date\today
	\maketitle

	\begin{abstract}
		For axisymmetric flows without swirl {and compactly supported initial vorticity,} we prove the upper bound of $t^{4/3}$ for the growth of the vorticity maximum, which was conjectured by Childress [Phys. D, 2008] and supported by numerical computations from Childress--Gilbert--Valiant [J. Fluid Mech. 2016]. The key is to estimate the velocity maximum by the kinetic energy together with conserved quantities involving the vorticity. 
	\end{abstract}
	
	\section{Introduction}\label{sec:intro}
	In this paper, we study the optimal rate of \textit{vortex stretching} for one of the simplest three-dimensional flows, namely axisymmetric flows without swirl. Let us first recall the 3D vorticity equation for incompressible and inviscid flows: for vorticity $\omg(t,\cdot):\bbR^{3}\to\bbR^{3}$, we have \begin{equation}\label{eq:Vorticityform}
		\left\{
		\begin{aligned}
			&\rd_{t}\omg+u\cdot \nb\omg=\omg \cdot \nb u,  \\
			&u=\nb\times(-\lap)^{-1}\omg,  \\
			&\omg|_{t=0} =\omg_{0}.
		\end{aligned}
		\right.
	\end{equation}  The term $\omg\cdot\nb u$ in the vorticity equation is referred to as the \textit{vortex stretching} term, and the Biot--Savart law $u=\nb\times(-\lap)^{-1}\omg$ is given explicitly in $\bbR^{3}$ by 
	\begin{equation}\label{eq:3DBSlaw}
		\begin{split}
			u(t,x)&=\int_{\bbR^{3}}K_{3}(x-y)\times\omg(t,y)dy , \quad 
			K_{3}(x)=\frac{1}{4\pi}\frac{x}{|x|^{3}},\quad x\in \bbR^{3}. 
		\end{split}
	\end{equation}  Let us reduce \eqref{eq:Vorticityform} by taking axisymmetric flows without swirl: this means to take the ansatz
	\begin{equation}\label{eq:ansatz}
		u=u^{r}(r,z)e^{r}+u^{z}(r,z)e^{z}, \qquad 		\omg=\omg^{\tht}(r,z)e^{\tht}=(-\rd_{z}u^{r}+\rd_{r}u^{z})e^{\tht}
	\end{equation}
	in the cylindrical coordinates $(e^r, e^\tht, e^z)$. Then, \eqref{eq:Vorticityform} simplifies to \begin{equation}\label{eq:asEuler}
		\begin{split}
			\rd_{t}\omg^{\tht}+ (u^r \rd_{r} + u^{z} \rd_{z})\omg^{\tht}&=\frac{u^{r}}{r}\omg^{\tht}, 
		\end{split}
	\end{equation} and note that the vortex stretching term has been reduced to $r^{-1}u^{r}\omg^{\tht}$ in \eqref{eq:asEuler}. Still, depending on the regularity of $\omg^{\tht}(t=0)$, this term can lead to either instantaneous, finite-time, or infinite-time \textit{infinite growth} of the vorticity maximum $\nrm{\omg^{\tht}(t, \cdot )}_{L^{\infty}}$ (\cite{Elgindi-3D,EGM,JK-axi,CJ-axi,GMT2023,CMZ}). 
	
	\subsection{Main result}\label{ssec:result}
	
	Our main result gives an upper bound for the growth rate of  $\nrm{\omg^{\tht}(t, \cdot )}_{L^{\infty}}$ for smooth data, which is believed to be optimal, see the discussion below. 
	\begin{thm}\label{thm:growth-optimal} Assume that $\omg^{\tht}_{0}$ is compactly supported and $\nrm{ r^{-1} \omg^{\tht}_{0} }_{L^\infty(\bbR^{3})}$ is bounded. Then the corresponding unique global-in-time solution $\omg(t,\cdot)$ of \eqref{eq:Vorticityform} belonging to $L^{\infty}_{loc}(\bbR;L^{\infty}(\bbR^3))$ with the initial data $\omg_{0} = \omg^{\tht}_{0} e^{\tht}$ satisfies,  for some constant $ A(\omg_{0}^{\tht})>0$ depending only on $\omg_{0}^{\tht}$,
		\begin{equation}\label{eq:growth3/2}
			\nrm{\omg(t, \cdot )}_{L^{\ift}(\bbR^{3})} \le A(\omg_{0}^{\tht}) (1 + |t|)^{4/3}\quad\text{for all}\quad t \in \bbR. 
		\end{equation}
	\end{thm}

	\begin{rem} Any smooth and compactly supported initial vorticity satisfies the assumption of Theorem \ref{thm:growth-optimal}; we have $\nrm{ r^{-1} \omg^{\tht}_{0} }_{L^\infty}<\infty$   once $\omg_{0} = \omg^{\tht}_{0}e^{\tht} \in C^{1}(\bbR^{3}).$ The constant $ A(\omg_{0}^{\tht})$ depends only on $\nrm{ \omg^{\tht}_{0} }_{L^\infty}$, $\nrm{ r^{-1} \omg^{\tht}_{0} }_{L^\infty}$, $\nrm{ r^{-1} \omg^{\tht}_{0} }_{L^1}$, and $\nrm{u_{0}}_{L^{2}}$. Since $\omg^{\tht}_{0}$ is compactly supported, $\nrm{ r^{-1} \omg^{\tht}_{0} }_{L^\infty}<\infty$ implies that $ \nrm{\omg_0^\tht}_{L^\infty}$ and $\nrm{r^{-1}\omg^\tht_{0}}_{L^p}$ (for any $p<\infty$) are bounded. Furthermore, from the fast decay of axisymmetric Biot--Savart kernel (see \S \ref{subsec:BS} below), the initial velocity belongs to $L^{2}(\bbR^{3})$. 
	\end{rem}

	\subsection{Wellposedness for axisymmetric Euler without swirl} In this section, let us clarify the wellposedness issue for \eqref{eq:asEuler} as well as \eqref{eq:Vorticityform}. A comprehensive survey of wellposedness results for \eqref{eq:Vorticityform} can be found in \cite{BL1,BL3D,DE,Elgindi-3D}.
	
	To begin with, global regularity of axisymmetric Euler without swirl \eqref{eq:asEuler} goes back at least to the work of Ukhovskii--Yudovich \cite{UY} in 1968, where global in time existence, uniqueness, propagation of regularity was shown assuming that \begin{equation*}
		\begin{split}
			\omg_{0} \in H^{m}(\bbR^{3}), \qquad \frac{\omg_{0}}{r} \in (L^{2} \cap L^{\infty})(\bbR^{3})
		\end{split}
	\end{equation*} with $m\ge2$. Such a regularity assumption on $r^{-1}\omg_{0}$ was natural, since \eqref{eq:asEuler} could be written as \begin{equation}\label{eq:asEuler-rel}
		\begin{split}
			\rd_{t}\, \frac{\omg^{\tht}}{r} + (u^r \rd_{r} + u^{z} \rd_{z})\, \frac{\omg^{\tht}}{r}  = 0. 
		\end{split}
	\end{equation} In particular, we have the a priori estimate \begin{equation*}
		\begin{split}
			\left\Vert \frac{\omg^{\tht}(t,\cdot)}{r} \right\Vert_{L^p(\bbR^3)} = 		\left\Vert \frac{\omg^{\tht}_{0}}{r} \right\Vert_{L^p(\bbR^3)}
		\end{split}
	\end{equation*} for any $0\le p \le \infty$ by incompressibility.\footnote{In the case $p=0$, it means the conservation of the volume of the support of $\omg^{\tht}$.} Therefore, this work already clarified that once the initial datum is sufficiently regular, the corresponding solution to \eqref{eq:asEuler} is regular and global. 
	
	On the other hand, local in time existence and uniqueness for the full 3D Euler equations \eqref{eq:Vorticityform} was obtained in H\"older spaces (\cite{Lichtenstein1930a,Holder,Wo}) in the 1930s, namely for $\omg_{0} \in C^{\alp}(\bbR^{3})$ with any $\alp>0$ (assuming some decay at infinity). Local regularity for $\omg_{0} \in H^{s}(\bbR^3)$ with any $s>3/2$ was obtained by Kato \cite{Ka} and for $W^{s,p}(\bbR^{3})$ with $sp>3$ by Kato--Ponce \cite{KaPo}. 
	
	These local regularity results guarantee that once $\omg_{0}$ belongs to such function spaces and takes the special form \eqref{eq:ansatz}, then the unique local solution to \eqref{eq:Vorticityform} solves \eqref{eq:asEuler}.\footnote{For the latter property to hold, one just requires uniqueness for \eqref{eq:Vorticityform} \textit{and} existence for \eqref{eq:asEuler} with certain regularity assumption. Uniqueness for \eqref{eq:Vorticityform} (but not existence) is guaranteed just with $\omg \in L^{\infty}_{t,x}$, see \cite{Danaxi}. Existence for \eqref{eq:Vorticityform} (but not uniqueness) is guaranteed simply with $r^{-1}\omg_{0}^{\tht} \in L^{1}\cap L^{p}$ with any $p>1$ \cite{JLN2}.} However, such a regularity assumption in general does not guarantee \textit{global} regularity for \eqref{eq:asEuler} and in particular, $H^{s}$ global regularity with $s>3/2$ was obtained later by Shirota--Yanagisawa \cite{SY} in 1994. (See also \cite{Raymond}.) 
	
	In two dimensions, the (scalar) vorticity is simply being advected by the flow, and the celebrated Yudovich theorem \cite{Yu63} gives global uniqueness and existence for vorticity belonging to $(L^\infty \cap L^1) (\bbR^2)$. Since the structure of \eqref{eq:asEuler-rel} is similar, it is tempting to obtain global regularity for \eqref{eq:asEuler} just based on $L^{p}$ regularity (i.e. without any differentiability) and  Danchin \cite{Danaxi} in 2007 achieved this with \begin{equation*}
		\begin{split}
			\omg_{0} \in (L^{3,1} \cap L^{\infty})(\bbR^{3}), \qquad \frac{\omg_{0}}{r} \in L^{3,1}(\bbR^{3}), 
		\end{split}
	\end{equation*} where $L^{3,1}(\bbR^{3})$ is a Lorentz space which strictly embeds into $L^{3}$. As a consequence, it gives global regularity for $\omg_{0} \in (C^{\alp}\cap L^{3,1})(\bbR^{3})$ with $\alp>1/3$. Abidi--Hmidi--Keraani (\cite{AHK}) extended global regularity to the critical Besov spaces $B^{3/p}_{p,1}(\bbR^{3})$ with all $p \in [1,\infty]$ in 2010.
	
	Recently, Elgindi \cite{Elgindi-3D} and Elgindi--Ghoul--Masmoudi \cite{EGM} proved finite time singularity formation of \eqref{eq:asEuler} for $\omg_{0}\in C^\alp(\bbR^{3})$ with sufficiently small $\alp>0$. Their unique local-in-time solution belonging to $L^{\infty}(0,T;C^{\alp}(\bbR^{3}))$ satisfies \begin{equation*}
		\begin{split}
			\lim_{t \to T^{-}} \nrm{\omg^\tht(t,\cdot)}_{L^{\infty}(\bbR^{3})} = \infty , 
		\end{split}
	\end{equation*} where $T$ is the singularity time. For a more recent improvement on finite-time blowup, see \cite{CMZ}. Moreover, \cite{JK-axi} proved that $\nrm{\omg^\tht(t,\cdot)}_{L^{\infty}} $ can \textit{instantaneously} blow up as soon as $t>0$ with compactly supported $\omg_{0}^{\tht} \in L^{\infty}(\bbR^{3})$ satisfying $r^{-1}\omg_{0}^{\tht} \in L^{3,q}(\bbR^{3})$ with all $q$ sufficiently large.\footnote{In particular this shows sharpness of Danchin's $L^{3,1}$ assumption. While it is expected that illposedness occurs for $L^{3,1+\varepsilon}$ for any $\varepsilon>0$, this could be a technically challenging problem.} In all these negative results, the \textit{lack of regularity} near the axis $\{ r = 0 \}$ is the driver of vorticity growth. 
	
	Lastly, even when the vorticity is regular near the axis, infinite-time infinite growth of the vorticity maximum for \eqref{eq:asEuler} is still possible: Very recently, some lower bounds have appeared (\cite{CJ-axi,GMT2023}), which establishes existence of data which completely vanishing near the axis but satisfies \begin{equation*}
		\begin{split}
			\nrm{\omg(t, \cdot )}_{L^{\ift}(\bbR^{3})} \ge c\,t^{\bt} \qquad \mbox{for all}\qquad t \ge 0 
		\end{split}
	\end{equation*} for some $c, \bt>0$. The value of $\bt$ can be as close to $3/8$ in \cite{GMT2023} and is smaller in \cite{CJ-axi}. The questions that we focus on this work are: what is the optimal rate of growth for such regular vorticities, and when is it achieved. Let us now proceed to review previous works in this direction.

	\begin{figure}
		\centering
		\includegraphics[scale=0.25]{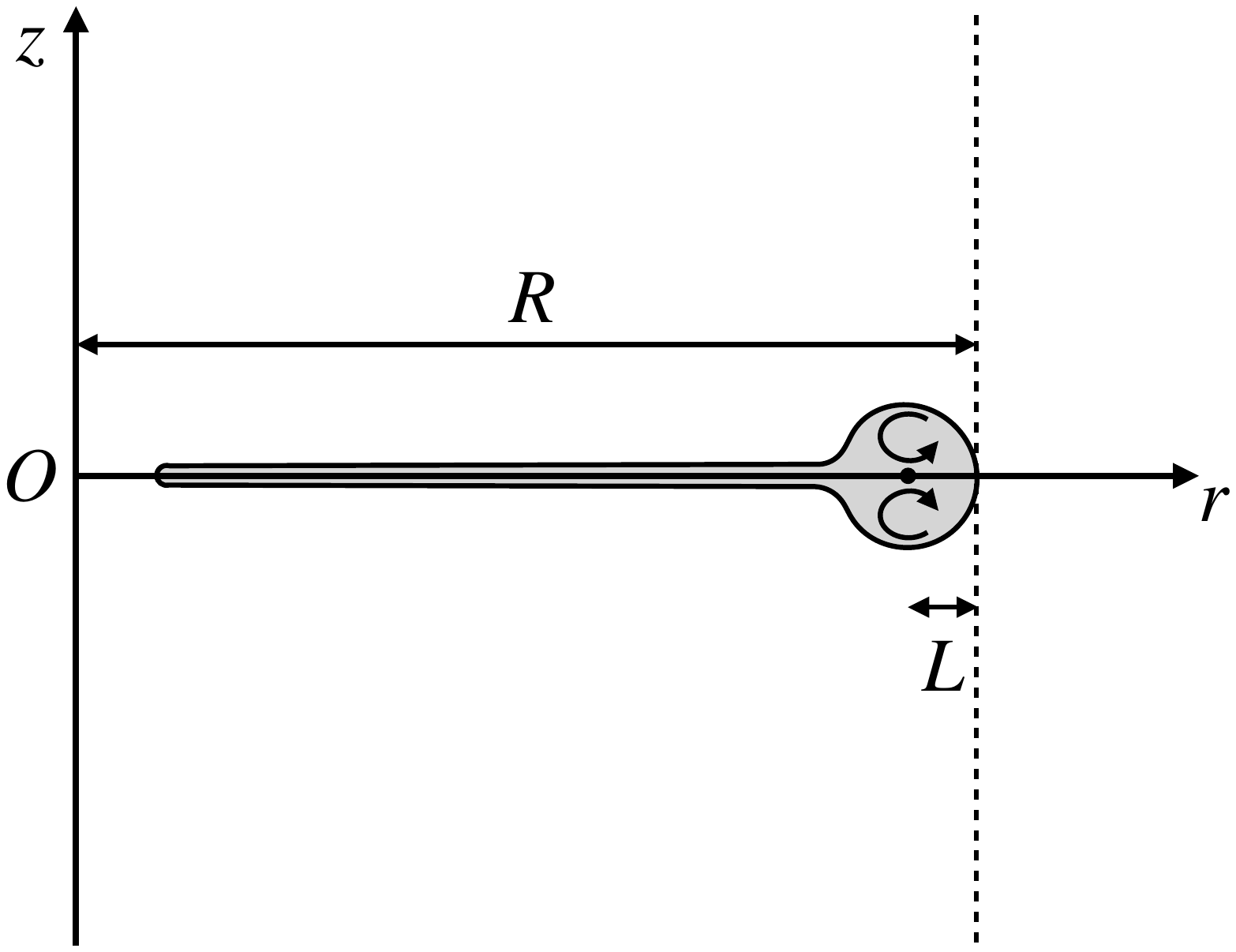}
		\caption{Vortex dipole with maximal support radius $ R $ and length scale $L$.} 
		\label{fig:dipole}
	\end{figure}

	\subsection{Optimal vorticity growth and eroding dipole model by Childress}
	
	It is well known to experimental physicists that vortex stretching occurs when two anti-parallel vortex rings move towards each other. This so-called ``head on collision'' scenario has been investigated intensively (\cite{Le,LN,ShLe,SSH,CWCCC}), where it was observed that the maximal vortex strength increases together with the Reynolds number. A simple but illuminating mathematical model of this scenario can be given based on \eqref{eq:asEuler} (formally taking the Reynolds number to be infinite), by taking the initial data to be of the form \begin{equation}\label{eq:rings}
		\begin{split}
			{\omg_{0}^{\tht}}(r,z) = - \varphi(r,z) + \varphi(r, -z),
		\end{split}
	\end{equation} where $\varphi \ge 0$ is a bump function supported near some point in $(r,z)$, say $(1,1)$. This models two anti-parallel vortex rings, and the key features are that the vorticity is odd symmetric in $z$ and non-positive for $z\ge0$. Under these two assumptions, it can be proved that (\cite{CJ-axi,GMT2023})\begin{equation*}
		\begin{split}
			\frac{d}{dt} \iint_{\Pi} r^{2}|\omg^{\tht}(t,r,z)| dr dz > 0, \qquad 	\frac{d}{dt} \iint_{\Pi} |z\omg^{\tht}(t,r,z)| dr dz < 0
		\end{split}
	\end{equation*} where $\Pi := \left\{ (r,z) : r \ge 0 \right\}$. This shows that the vorticity moves southeast in the $(r,z)$-plane in average. Actually, direct numerical simulations of \eqref{eq:asEuler} reveals that for initial data of the type \eqref{eq:rings}, the solution $\omg^{\tht}(t,\cdot)$ rather quickly attaches to the plane $\{ z = 0 \}$ as it slides outwards in the $r$-axis (\cite{Carley,Shariff.Leonard.Zabusky.Ferziger.1988,Shariff.Leonard.Ferziger.2008}). Furthermore, in all of the aforementioned vorticity growth results (\cite{Elgindi-3D,EGM,JK-axi,CJ-axi,GMT2023,CMZ}), the data indeed share the same symmetry and sign condition with \eqref{eq:rings}: it is essential to have $u^{r} > 0$ for vorticity growth. 
	
	This is precisely the setup that Childress and his collaborators investigated in detail (\cite{Child07,Child08,ChilGil1,ChilGil2}). In \cite[Theorem 1]{Child07}, Childress proves the upper bound\footnote{The initial data can be general and not restricted to the form given in \eqref{eq:rings}.} \begin{equation}\label{eq:t2}
		\begin{split}
			\nrm{\omg^{\tht}(t,\cdot)}_{L^{\infty}(\bbR^{3})} \le C(1+t^2),
		\end{split}
	\end{equation}  by solving the following variational optimization problem: at time $t$, assuming that $\omg^{\tht}$ is supported near some radius $R \gg 1$, what is the optimal shape of $\omg^{\tht}$ which maximizes its radial velocity $u^{r}$, under a constraint on the volume of the support of $\omg^{\tht}$ (which is conserved in time)? 
	
	Following \cite{Child07,Child08}, $t^{2}$ can be easily explained as follows: assume that the vorticity at time $t$ is localized near the point $(R(t),0)$ with $L(t)$ as its own length scale. Assuming that $L(t) \ll R(t)$, the velocity on the support of vorticity scales as $R(t)L(t)$: {roughly speaking, this is because the velocity can be obtained by ``integrating'' the vorticity (which has size $\sim R(t)$) over a domain with length scale $\sim L(t)$.} Then the volume conservation imposes the restriction $R(t)L^{2}(t) \lesssim 1$, or $L(t) \lesssim (R(t))^{-1/2}$. Then, from \begin{equation*}
		\begin{split}
			\frac{d}{dt} R(t)  \sim \frac{d}{dt}\nrm{\omg^{\tht}(t,\cdot)}_{L^{\infty}(\bbR^{3})} \lesssim \nrm{u^{r}(t,\cdot)}_{L^{\infty}(\supp( \omg^{\tht}(t,\cdot) ))} \sim R(t)L(t) \lesssim (R(t))^{1/2}, 
		\end{split}
	\end{equation*} we arrive at \eqref{eq:t2}. 
	
	Furthermore, in the same papers, Childress makes an important observation that this scenario contradicts the conservation of kinetic energy, namely $\nrm{u}^{2}_{L^{2}%(\bbR^{3})
	}$. Indeed, using that the support volume scales like $R(t)L^{2}(t)$ and the velocity $R(t)L(t)$, we have \begin{equation*}
		\begin{split}
			\nrm{u}^{2}_{L^{2}(\bbR^{3})} \gtrsim R^{3}(t)L^{4}(t), 
		\end{split}
	\end{equation*} and the scaling $L(t) \sim (R(t))^{-1/2}$ clearly makes the right hand side divergent as $R(t)\to\infty$. To fix this, we need to have $L(t) \lesssim (R(t))^{-3/4}$, which gives instead \begin{equation*}
		\begin{split}
			\frac{d}{dt} R(t) \lesssim (R(t))^{1/4}. 
		\end{split}
	\end{equation*} Based on these considerations, Childress conjectured in \cite{Child07,Child08} that the maximal growth rate is given by $t^{4/3}$, which is confirmed in our Theorem \ref{thm:growth-optimal}. Later, Childress--Gilbert--Valiant \cite{ChilGil1} presents numerical experiments which supports that the growth rate of $t^{4/3}$ can be indeed achieved. They also report that the optimal growth is attained by rescaled Sadovskii vortex patch--see \cite{CJS,HT} for more discussion related to this point. 
	
	Finally, we observe that $L(t) \lesssim (R(t))^{-3/4}$ implies that the volume of the dipole support in $\bbR^{3}$, which is $R(t)L^{2}(t) \lesssim (R(t))^{-1/2}$, goes to zero as $R(t)\to\infty.$ Since the total volume of the vorticity support remains constant, it means that most of the volume filaments behind (this is what we have attempted to depict in Figure \ref{fig:dipole}). Hence the terminology ``eroding'' dipole in \cite{Child07,Child08} is appropriate.

	\subsection{Velocity bounds for axisymmetric flows: the proof strategy} 
	Let us now review existing bounds on the velocity  for axisymmetric flows and then discuss our proof strategy. 
	
	\subsubsection{Exponential upper bound} To begin with, it is relatively straightforward to obtain the exponential in time upper bound for the vorticity. This can be done in a number of ways: in \cite{Danaxi}, the following estimate was established: \begin{equation*}
		\begin{split}
			\left\Vert \frac{u^r}{r} \right\Vert_{L^{\infty}(\bbR^{3})} \lesssim 	\left\Vert \frac{\omg^{\tht}}{r} \right\Vert_{L^{3,1}(\bbR^{3})}. 
		\end{split}
	\end{equation*} Using that the right hand side is conserved in time, one obtains the vorticity bound from \eqref{eq:asEuler} \begin{equation*}
		\begin{split}
			\frac{d}{dt}\nrm{\omg^{\tht}(t,\cdot)}_{L^{\infty}(\bbR^{3})} \lesssim 	\left\Vert \frac{u^r(t,\cdot)}{r} \right\Vert_{L^{\infty}(\bbR^{3})} \nrm{\omg^{\tht}(t,\cdot)}_{L^{\infty}(\bbR^{3})} \lesssim  \nrm{\omg^{\tht}(t,\cdot)}_{L^{\infty}(\bbR^{3})} 
		\end{split}
	\end{equation*} which results in an exponential in time upper bound, using Gronwall's inequality. Another estimate (assuming compact support of $\omg^{\tht}$) from Majda--Bertozzi \cite{MB} is $	\nrm{u(t,\cdot)}_{L^{\infty}} \lesssim  \nrm{\omg^{\tht}(t,\cdot)}_{L^{\infty}} $, which just uses the conservation of the volume of the support for $\omg^{\tht}$. This again results in an exponential upper bound for $ \nrm{\omg^{\tht}(t,\cdot)}_{L^{\infty}} $. To the best of our knowledge, before the work of Childress, the exponential bound was still the best one. 
	
	\subsubsection{Feng--Sverak estimate and the $t^{2}$ bound again}  
	
	A very interesting inequality for axisymmetric flows by Feng--Sverak \cite{FeSv} in 2015 gives an alternative, simple proof of the $t^{2}$ upper bound (although the authors had a completely different motivation). Their inequality reads 	\begin{equation}\label{eq:FS}
		\nrm{u }_{L^{\ift}(\bbR^{3})} \lesssim \bigg\|\frac{\omg^{\tht}}{r}\bigg\|_{L^{\ift}(\bbR^{3})}^{1/2}\bigg\|\frac{\omg^{\tht}}{r}\bigg\|_{L^{1}(\bbR^{3})}^{1/4}\nrm{r\omg^{\tht}}_{L^{1}(\bbR^{3})}^{1/4}
	\end{equation} (where $u = (u^r, u^z)$) and given this inequality, one can estimate \begin{equation*}
		\begin{split}
			\nrm{r\omg^{\tht}(t,\cdot)}_{L^{1}(\bbR^{3})}^{1/4} \le (R(t))^{1/2} \bigg\|\frac{\omg^{\tht}(t,\cdot)}{r}\bigg\|_{L^{1}(\bbR^{3})}^{1/4}
		\end{split}
	\end{equation*}and with conservation of $\nrm{r^{-1}\omg^{\tht}}_{L^p}$, one obtains \begin{equation*}
		\begin{split}
			\frac{d}{dt} R(t) \lesssim (R(t))^{1/2} ,
		\end{split}
	\end{equation*} which immediately gives the $t^{2}$ upper bound.\footnote{Let us clarify that the quantity $\nrm{r\omg^{\tht}}_{L^{1}%(\bbR^{3})
		}$ is conserved only when $\omg^{\tht}$ is \textit{single-signed}, namely if it is non-negative or non-positive everywhere in $\Pi$. Indeed, for single-signed vorticity, one concludes from \eqref{eq:FS} that the velocity is uniformly bounded in space and time, which gives the linear bound $R(t) \lesssim 1+ t$. Actually, in the single-signed case, the work of Maffei--Marchioro \cite{Maffei2001} in 2001 provides an upper bound of the form $R(t) \lesssim 1 + t^{1/4}\log(e+t)$.}
	
	Actually there is a reason why a sharp inequality on $\nrm{u}_{L^{\infty}}$ just involving $L^{p}$ norms of $\omg^{\tht}$ (possibly with $r$ weights) results in the same answer with Childress' variational proof. This requires an understanding of how the exponents in the right hand side of \eqref{eq:FS} are determined. This is non-trivial since three norms appear, and for the inequality \begin{equation}\label{eq:FS-hyp}
		\begin{split}
			\nrm{u }_{L^{\ift}(\bbR^{3})} \lesssim \bigg\|\frac{\omg^{\tht}}{r}\bigg\|_{L^{\ift}(\bbR^{3})}^{\alp}\bigg\|\frac{\omg^{\tht}}{r}\bigg\|_{L^{1}(\bbR^{3})}^{\bt}\nrm{r\omg^{\tht}}_{L^{1}(\bbR^{3})}^{\gmm}
		\end{split}
	\end{equation} to hold, we clearly need \begin{itemize}
		\item $0 \le \alp, \bt, \gmm$ and $\alp+\bt+\gmm = 1$ (homogeneity)
		\item $2\alp - \bt - 3\gmm = 0$ (3D scaling law). 
	\end{itemize} We need one more relation to determine the exponents, which comes from \textit{two-dimensionalization}: for \eqref{eq:FS-hyp} to hold, it should be true in particular when $\omg^{\tht}$ is sharply concentrated near a single point in $\Pi$, say $(1,0)$. In that case, \eqref{eq:FS-hyp} simply reduces to \begin{equation}\label{eq:FS-hyp-red}
		\begin{split}
			\nrm{u }_{L^{\ift}(\bbR^{3})} \lesssim  \| {\omg^{\tht}}  \|_{L^{\ift}(\bbR^{3})}^{\alp} \nrm{\omg^{\tht}}_{L^{1}(\bbR^{3})}^{{\bt+\gmm}}, 
		\end{split}
	\end{equation} and the punchline is that for such concentrated vortices, the 3D Biot--Savart kernel \eqref{eq:3DBSlaw} is asymptotic to the 2D Biot--Savart law. Therefore, \eqref{eq:FS-hyp-red} must obey the 2D scaling law, which is \begin{itemize}
		\item {$\alp=1/2=\bt+\gmm$}. 
	\end{itemize}
	This last part was indeed the essential observation in the $t^2$ proof of Childress: the velocity scales like $R(t)L(t)$, exactly as in 2D fluids.

	\subsubsection{Global velocity bound using kinetic energy and the $t^{3/2}$ bound} Given that the kinetic energy conservation plays the key role in the $t^{4/3}$ conjecture of Childress, it was very natural for us to attempt to get a Feng--Sverak type inequality which features $\nrm{u}_{L^{2}}$ in the right hand side. After a lot of trial and error, we could finally prove the following inequality (see Proposition \ref{prop:estimate} below) \begin{equation}\label{eq:newest}
		\nrm{u^{r}}_{L^{\ift}(\bbR^{3})} \lesssim \nrm{u}_{L^{2}(\bbR^{3})}^{1/3}\bigg\|\frac{\omg^{\tht}}{r}\bigg\|_{L^{\ift}(\bbR^{3})}^{1/2}\nrm{r\omg^{\tht}}_{L^{1}(\bbR^{3})}^{1/6}
	\end{equation}
	where the exponents on the right hand side can be explained as in the case of \eqref{eq:FS}. Now that the exponent of $\nrm{r\omg^{\tht}}_{L^{1}}$ is lower, it gives a better bound which is $t^{3/2}$. 	Unfortunately, this is still worse than $t^{4/3}$, and we were unable to lower the exponent $1/6$ in \eqref{eq:newest}.\footnote{An elementary algebraic computation shows that $1/6$ is the smallest power of $\nrm{r\omg^{\tht}}_{L^{1}}$ that one can use, under the assumption that on the right hand side, one is using $\nrm{r\omg^{\tht}}_{L^{1}}$ together with  $\nrm{u}_{L^2}$ and $\nrm{r^{-1}\omg^{\tht}}_{L^p}$ for all possible values of $p$.} 
	
	{Before we proceed to explaining the proof for $t^{4/3}$ bound, we note that the $t^{3/2}$ upper bound is the best we can obtain \textit{without} the compact support assumption for the vorticity. We still need some decay, which is $r\omg_{0}^{\tht} \in L^{1}(\bbR^{3})$. 
		 \begin{thm}\label{thm:t3/2}
			Assume that $\omg_{0}^{\tht}$ and $ r^{-1}\omg_{0}^{\tht} $ are in $(L^{1}\cap L^{\ift})(\bbR^{3})$ {and that $r\omg_{0}^{\tht}$ is in $L^{1}(\bbR^{3})$.} Then the corresponding unique global-in-time solution $\omg(t,\cdot)$ of \eqref{eq:Vorticityform} belonging to $L_{loc}^{\ift}(\bbR;(L^{1}\cap L^{\ift})(\bbR^{3}))$ with the initial data $\omg_{0}=\omg_{0}^{\tht}e^{\tht}$ satisfies {$r\omg \in L_{loc}^{\ift}(\bbR;L^{1}(\bbR^{3}))$ and}, for some constant $A(\omg_{0}^{\tht})>0$ depending only on $\omg_{0}^{\tht}$,
			\begin{equation}\label{eq:t3/2est}
				\nrm{\omg(t,\cdot)}_{L^{\ift}(\bbR^{3})}\leq A(\omg_{0}^{\tht})(1+|t|)^{3/2}\quad\text{for all}\quad t\in\bbR.
			\end{equation}
	\end{thm}
	We shall prove this result in Appendix \ref{sec:growth3/2} after the proof of Proposition \ref{prop:estimate}.
	}

	\subsubsection{Strategy and heuristic for the $t^{4/3}$ bound} At a heuristic level, it is easy to see why \eqref{eq:newest} does not give the optimal upper bound: in the eroding dipole model of Childress, the quantity $\nrm{r\omg^{\tht}}_{L^{1}}$ measured \textit{locally} on the leading dipole should decay in time due to filamentation. The question was how to make this observation precise and obtain the $t^{4/3}$ bound. 
	
	To begin with, while the inequality \eqref{eq:newest} is interesting in its own, it is wasteful in the sense that for the growth of the vorticity maximum, we only need to obtain a bound on $u^{r}(t,\cdot)$ at the maximal support radius of $\omg^{\tht}(t,\cdot)$, which we denote by $R(t)$. 
	
	The support of vorticity is contained in the ``rectangle'' $(r,z) \in [0,R(t)] \times \bbR$, and it turns out to be useful to divide the rectangle into two regions: \textit{axial regime} ($(I_1)$ in Figure \ref{fig:region}) and \textit{two-dimensional regime} ($(I_{21}) \cup (I_{22})$ in Figure \ref{fig:region}). 
	
	First, estimating the contribution to $u^{r}$ coming from the axial regime is easier for two reasons: this region is relatively far away from the point $(R(t),0)$ and the axisymmetric Biot--Savart kernel enjoys enhanced decay relative to the two-dimensional one (see Lemma \ref{lem:calFest}). Furthermore, we have bounds on $\nrm{r^{-1}\omg^{\tht}}_{L^p}$-norms, which says that the vorticity is small for $r$ small. Based on these, we were able to bound the contribution to $\nrm{u^r}_{L^\infty}$ coming from the axial regime simply by \begin{equation*}
		\begin{split}
			\bigg\|\frac{\omg^{\tht}}{r}\bigg\|_{L^{\ift}(\bbR^{3})}^{1/3}  \bigg\|\frac{\omg^{\tht}}{r}\bigg\|_{L^{1}(\bbR^{3})}^{2/3}
		\end{split}
	\end{equation*} which is conserved in time: this regime cannot contribute to any superlinear growth of $R(t)$.
	
	Therefore, the key was to understand the two-dimensional regime in detail. For this purpose, it is actually very helpful to consider the two-dimensional Euler equations. Denoting $\omg_{2D}:\bbR^2\to\bbR$ as the scalar vorticity on the plane and $u_{2D}$ the corresponding velocity given by \begin{equation*}
		\begin{split}
			u_{2D}(x) = \frac{1}{2\pi} \iint_{\bbR^{2}} \frac{(x-y)^\perp}{|x-y|^{2}} \omg_{2D}(y) dy, \qquad x \in \bbR^{2}, 
		\end{split}
	\end{equation*} as a warm up one can actually prove the following: \begin{equation*}
		\begin{split}
			\nrm{ u_{2D} }_{L^{\infty}(\bbR^2)} \lesssim 	\nrm{ u_{2D} }_{L^{2}(\bbR^2)}^{1/2} \nrm{\omg_{2D} }_{L^{\infty}(\bbR^2)}^{1/2}. 
		\end{split}
	\end{equation*} Believing in this inequality and two-dimensionalization of the axisymmetric Biot--Savart kernel, one may expect to have something like \begin{equation*}
		\begin{split}
			\nrm{ u^{r} }_{L^{\infty}(U)} \lesssim 	\nrm{ u }_{L^{2}(U)}^{1/2} \nrm{\omg^{\tht} }_{L^{\infty}(U)}^{1/2}
		\end{split}
	\end{equation*} where $U \subset \Pi $ is a small ball centered at $(R(t),0)$. Now to turn the $L^{2}(U)$ norm into the $L^{2}$ norm measured in $\bbR^{3}$, we need to incorporate the metric factor: \begin{equation*}
		\begin{split}
			\nrm{ u }_{L^{2}(U)}^{1/2} \sim	R^{-1/4}\nrm{ r^{1/2} u }_{L^{2}(U)}^{1/2} \lesssim 	R^{-1/4}\nrm{   u }_{L^{2}(\bbR^{3})}^{1/2} . 
		\end{split}
	\end{equation*} Moreover, noting that $\nrm{\omg^{\tht} }_{L^{\infty}(U)}^{1/2} \sim R^{1/2} \nrm{r^{-1}\omg^{\tht} }_{L^{\infty}(U)}^{1/2}  \lesssim R^{1/2}\nrm{r^{-1}\omg^{\tht} }_{L^{\infty}(\bbR^{3})}^{1/2}$, one obtains \begin{equation*}
		\begin{split}
			\nrm{ u }_{L^{2}(U)}^{1/2} \nrm{\omg^{\tht} }_{L^{\infty}(U)}^{1/2} \lesssim R^{1/4} \nrm{   u }_{L^{2}(\bbR^{3})}^{1/2}  \nrm{r^{-1}\omg^{\tht} }_{L^{\infty}(\bbR^{3})}^{1/2},
		\end{split}
	\end{equation*} which says that the contribution coming from the two-dimensional regime is bounded by $R^{1/4}$.

	Therefore, if one could justify all these heuristic discussions, one would obtain the inequality \begin{equation*}
		\begin{split}
			\frac{d}{dt} R(t) \sim \sup_{z\in\bbR} \, u^{r}(t,R(t),z) \lesssim 1 + (R(t))^{1/4} ,
		\end{split}
	\end{equation*} which finally gives the $t^{4/3}$ bound. The formal statement and proof for this key inequality are given in \S \ref{subsec:key} below. 
	
	\subsection{Growth rates in higher dimensions and related models}\label{subsec:highD} 
	
	Our new bounds on the axisymmetric velocity can be readily extended to the case of higher dimensional axisymmetric Euler equations, which were studied for instance in \cite{Yang,KhYa,CJLglobal22,Limglobal23,GMT2023}.  
	
	One can compute that the $d$-dimensional Euler equations under the assumption of axisymmetric without swirl take the form  
	\begin{equation}\label{eq:asEulerhighd}
		\begin{split}
			\rd_{t}\omg + (u^r \rd_{r} + u^{z} \rd_{z})\omg &=\frac{(d-2)u^{r}}{r}\omg 
		\end{split}
	\end{equation} with a scalar vorticity $\omg$ for any $d\ge3$. 
	\begin{thm}\label{thm:growth-optimalhighd} Assume that $\omg_{0}$ is compactly supported and $\nrm{ r^{-(d-2)} \omg_{0} }_{L^\infty(\bbR^{d})}$ is bounded. 
		Then, for $3 \le d \le 6$, the corresponding local-in-time unique solution of \eqref{eq:asEulerhighd} is actually {global-in-time} and  satisfies
		\begin{equation}\label{eq:growth4/3highd}
			\nrm{\omg(t,\cdot )}_{L^{\ift}(\bbR^{d})} \le \begin{cases}
				A_{1}(1+|t|)^{4/(6-d)},&\quad d=3,4,5,\\
				A_{2}e^{A_{3}|t|},&\quad d=6,
			\end{cases}\quad\text{for all}\quad t\in\bbR 
		\end{equation} for some constants $ A_{1}(d,\omg_{0}), A_{2}(\omg_{0}), A_{3}(\omg_{0})>0$. %depending only on $\omg_{0}$. 
	\end{thm}
	
	This seems to be the first global regularity result for dimensions 5 and 6, without a sign assumption on the vorticity. With the single sign assumption, global regularity was proved in 	\cite{Limglobal23} for any $d\ge3$. Without the sign assumption, global regularity for $d = 4$ was obtained in  \cite{CJLglobal22}, under the assumption $\nrm{r^{-2}\omg_{0}}_{L^\infty}<\infty$.\footnote{Note that unfortunately this assumption is not automatically satisfied by simply assuming sufficient smoothness on the vorticity. The situation is the same for higher dimensions.} However, in  \cite{CJLglobal22}, the upper bound on $	\nrm{\omg(t,\cdot )}_{L^{\ift}%(\bbR^{4})
	} $ was exponential in time.

	Lastly, we expect our new estimates to be useful for some other models, including the lake equations (\cite{LOT96,AlLa}), bi-rotational Euler equations (\cite{CJL_gwpbi}) and high dimensional axisymmetric Navier--Stokes equations. 
	
	\subsection{Organization of the paper} The rest of this paper is organized as follows. In \S\ref{sec:prelim}, we first collect notations and recall estimates for the axisymmetric Biot--Savart kernel. The main results are proved in \S \ref{sec:proof}. In \S \ref{subsec:key}, we prove the key $R^{1/4}$ estimate on the velocity, which allows us to conclude Theorem \ref{thm:growth-optimal} in \S \ref{sec:main}. Theorem \ref{thm:growth-optimalhighd} which concerns higher dimensional case is proved in \S \ref{sec:mainhighd}. In the appendix, we include the proof of the global velocity estimate \eqref{eq:newest}, which could be of its own interest.

	\subsection*{Acknowledgments}
	I.-J. Jeong was supported by Samsung Science and Technology Foundation under Project Number SSTF-BA2002-04. D. Lim was supported by the National Research Foundation of Korea grant RS-2024-00350427. We sincerely thank the anonymous referee for the careful reading of our manuscript and for their valuable comments, and in particular for encouraging us to make Theorem \ref{thm:t3/2} into a separate statement. We also thank Kwan Woo for finding a few typos in the earlier version.

	\medskip

	\section{Preliminaries}\label{sec:prelim}

	\subsection{Notations}\label{sec:not} We collect the notations used throughout the paper. 
	\begin{itemize}
		
		\item It is convenient to introduce the half-plane $\Pi := \left\{ (r, z) : r \ge 0, z \in \bbR \right\}$, as $\omg^{\tht}$, $u^{r}$, $u^{z}$ are functions defined on $\Pi$. 
		
		\item Let us write $\omg$ instead of $\omg^{\tht}$ from now on. 
		
		\item We use the letter $R$ to denote the maximal support radius of $\omg$, namely $R$ is the smallest non-negative number satisfying $\omg(r,z)=0$ whenever $r>R$. We write $R(t)$ when $\omg$ is time dependent.
		
		\item We write ${D_{r,z}(\br{r},\br{z})=}D(r,z,\bar{r},\bar{z}) 
		:= ((r-\br{r})^{2}+(z-\br{z})^{2})^{1/2}$ as the distance between the points $(r,z)$ and $(\bar{r},\bar{z})$ in $\Pi$. 
		
		{\item For any $k>0$ and any point $(r,z)\in \Pi$, we denote the disc centered at the point $ (r,z) $ with radius $ k $ as
			$$ B_{k}(r,z):=%\lbrace D_{r,z}<k\rbrace= 
			\lbrace(\br{r},\br{z})\in\Pi : D(r,z,\br{r},\br{z})<k\rbrace. $$
			\item We denote $ \chi\in C_{c}^{\ift}([0,\ift)%\bbR_{\geq0}
			;[0,1]) $ as a non-increasing function that satisfies $ \chi\equiv1 $ on $ [0,1/2] $ and $ \chi\equiv0 $ on $ (1,\ift) $. }
		
		\item For non-negative expressions $A$ and $B$, we write $A \lesssim B$ if there is an absolute constant $C>0$ such that $A \le CB$. We write $A \sim B$ if $A \lesssim B$ and $B \lesssim A$. 
		
	\end{itemize}

	\subsection{Axisymmetric Biot--Savart law }\label{subsec:BS} We recall a very convenient form of the axisymmetric Biot--Savart law derived in Feng--Sverak \cite{FeSv}. To begin with, we introduce the elliptic integral \begin{equation}\label{eq:calF}
		\begin{split}
			\calF(s):=\int_{0}^{\pi}\frac{\cos\alp}{[2(1-\cos\alp)+s]^{1/2}}d\alp,\quad s>0.
		\end{split}
	\end{equation} Then, under axisymmetry, 
	\eqref{eq:3DBSlaw} can be represented in the form 
	\begin{equation}\label{eq:asBiotSavart_ur}
		\begin{split}
			u^{r}(r,z) = \iint_{\Pi}F^{r}(r,z,\br{r},\br{z})\omg^{\tht}(\br{r},\br{z})d\br{z}d\br{r}, \qquad F^{r}(r,z,\br{r},\br{z})=  \frac{-(z-\br{z})}{\pi r^{3/2}\br{r}^{1/2}}\calF'\bigg(\frac{D^{2}}{r\br{r}}\bigg) 
		\end{split}
	\end{equation} 
	where we recall from \S \ref{sec:not} that  $\Pi=\lbrace(r,z) %\in \bbR^{2} 
	: r\geq0\rbrace$ and $D(r,z,\bar{r},\bar{z}) = ((r-\br{r})^{2}+(z-\br{z})^{2})^{1/2}$. Similarly, one can write \begin{equation*}
		\begin{split}
			u^{z}(r,z)=  \iint_{\Pi}F^{z}(r,z,\br{r},\br{z})\omg^{\tht}(\br{r},\br{z})d\br{z}d\br{r}
		\end{split}
	\end{equation*} where \begin{equation*}
		\begin{split}
			&F^{z}(r,z,\br{r},\br{z})=\frac{r-\br{r}}{\pi r^{3/2}\br{r}^{1/2}}\calF'\bigg(\frac{D^{2}}{r\br{r}}\bigg)  + \frac{\br{r}^{1/2}}{4\pi r^{3/2}}\bigg[\calF\bigg(\frac{D^{2}}{r\br{r}}\bigg)- \frac{2 D^{2}}{r\br{r}}\calF'\bigg(\frac{D^{2}}{r\br{r}}\bigg)\bigg]. 
		\end{split}
	\end{equation*} 
	Later, we shall obtain bounds on the kernel $F^{r}$ and its derivatives by the following result.
	\begin{lem}[{{\cite[Corollary 2.9]{FeSv}}}]\label{lem:calFest}
		For any integer $\ell \ge 1$, the $ \ell $-th derivative of $ \calF $ satisfies
		\begin{align} 
			|\calF^{(\ell)}(s)|&\lesssim_{\ell}\min\lbrace s^{-\ell}, s^{-(\ell+3/2)}\rbrace .\label{eq:calFkest}
		\end{align}
	\end{lem}
	
	\begin{rem}\label{rem:F}
		Although we shall not use this fact, one has $|\calF(s)| \lesssim \min\{{|\log(s)|+1}, s^{-3/2}\}$. More precisely, for $0<s\ll1$, we have that $\calF(s) = c_{0}\log(s) + O(1)$, $\calF'(s) = c_{0}s^{-1} + O(s^{-2})$ and so on, for an absolute constant $c_{0}\ne0$. 
	\end{rem}

	\medskip 
	
	\section{Proof of the main result}\label{sec:proof}
	
	\subsection{The $R^{1/4}$ estimate}\label{subsec:key}
	
	In this section, we prove the key estimate, which bounds the radial velocity \textit{on} the maximal support radius of the vorticity. 
	
	\begin{proposition}\label{prop:urRz}
		Let $\omg$ be compactly supported with maximal radius of the support $R>0$. Then  we have
		\begin{equation}\label{eq:uratRz}
			\sup_{z\in\bbR}|u^{r}(R,z)|\lesssim\bigg\|\frac{\omg}{r}\bigg\|_{L^{\ift}(\bbR^{3})}^{1/3}\bigg\|\frac{\omg}{r}\bigg\|_{L^{1}(\bbR^{3})}^{2/3}+R^{1/4} \nrm{u}_{L^{2}(\bbR^{3})}^{1/2}\bigg\|\frac{\omg}{r}\bigg\|_{L^{\ift}(\bbR^{3})}^{1/2}.
	\end{equation}	\end{proposition}
	
	\begin{proof}
		We fix some $z\in \bbR$, and  
		define using $\chi$ from \S \ref{sec:not} the rescaled cutoff functions
		\begin{equation}\label{eq:Phi}
			\Phi_{k}(\br{r},\br{z}):= \chi\bigg(\frac{ D({R,z,}\br{r},\br{z})  }{k}\bigg),\quad k>0.
		\end{equation}

		To begin with, using $k=R/2$, we split $u^{r}(R,z)$ into
		\begin{equation}\label{eq:urRzsplit}
			\begin{split}
				&\ \underbrace{\iint_{\Pi}F^{r}(R,z,\br{r},\br{z})\big(1-\Phi_{R/2}(\br{r},\br{z})\big)\omg(\br{r},\br{z})d\br{z}d\br{r}}_{=:(I_{1})} +\underbrace{\iint_{\Pi}F^{r}(R,z,\br{r},\br{z})\Phi_{R/2}(\br{r},\br{z})\omg(\br{r},\br{z})d\br{z}d\br{r}}_{=:(I_{2})}. 
			\end{split}
		\end{equation}
		To ease notation, let us abbreviate integrals $ (I_{1}) $ and $ (I_{2}) $ as
		$$ (I_{1})=\iint_{\Pi}F^{r}(1-\Phi_{R/2})\omg,\quad (I_{2})=\iint_{\Pi}F^{r}\Phi_{R/2}\omg. $$
		
		\begin{figure} 
			\centering
			\includegraphics[scale=0.25]{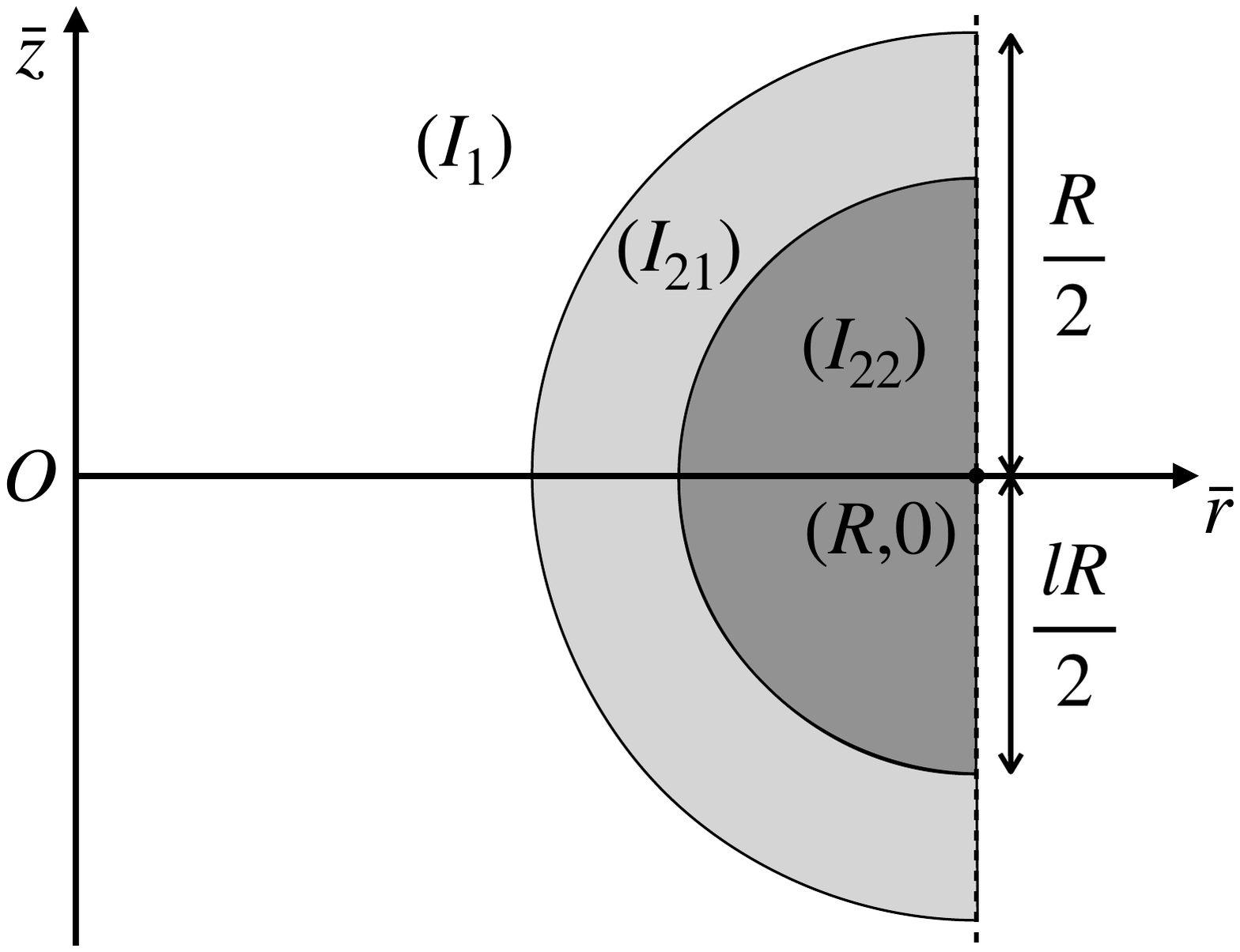}
			\caption{A diagram showing of $ (I_{1}) $ from \eqref{eq:urRzsplit} and $ (I_{21}) $ %, $ (I_{22}) $ 
				from \eqref{eq:I2split} when $ l<1 $, with $z = 0$. The darkest region represents $(I_{22})$, which is $(I_{2}) $ minus $ (I_{21}) $.}  
			\label{fig:region}
		\end{figure}
		
		Below, we shall further split $(I_{2})$ into $(I_{21})$ and $(I_{22})$, see Figure \ref{fig:region}. 
		
		\medskip 
		
		\noindent \textbf{Estimate of $ (I_{1}) $}. In the term $ (I_{1}) $, the integrand is supported on the region $ \Pi\setminus B_{R/4}(R,z) $. To estimate this term, we use the estimate
		\begin{equation}\label{eq:FrRznonloc}
			\begin{split}
				|F^{r}(R,z,\br{r},\br{z})|&\lesssim\frac{|\br{z}-z|}{R^{3/2}\br{r}^{1/2}} \frac{(R\br{r})^{5/2}}{D_{{R,z}}^{5}}\leq\frac{R\br{r}^{2}}{ D_{{R,z}}^{4} },
			\end{split}
		\end{equation}
		which holds on the region $\Pi\setminus B_{R/4}(R,z)$. 
		Then we have, using $\br{r} \le  R  \le 4 D_{{R,z}}, $\footnote{This is the only place where we really need that $R$ is the maximal support radius of $\omg$.}
		\begin{align}
			|(I_{1})|&\lesssim\iint_{{\Pi\setminus B_{R/4}(R,z)}}\frac{R\br{r}^{2} %|1-\Phi_{R/2}|
			}{D_{{R,z}}^{4}}  |\omg|^{2/3} |\omg|^{1/3} \frac{\br{r}^{1/3}}{\br{r}^{1/3}}\lesssim R^{1/3} \bigg\|\frac{\omg}{r}\bigg\|_{L^{\ift}({\Pi})}^{1/3} \iint_{{\Pi\setminus B_{R/4}(R,z)}}\frac{1%|1-\Phi_{R/2}|
			}{D_{{R,z}}}  |\omg|^{2/3}\nonumber\\
			&\leq R^{1/3} \bigg\|\frac{\omg}{r}\bigg\|_{L^{\ift}({\Pi})}^{1/3} \bigg(\iint_{{\Pi\setminus B_{R/4}(R,z)}}\frac{1%|1-\Phi_{R/2}|^{3}
			}{D_{{R,z}}^{3}} \bigg)^{1/3} \bigg(\iint_{{\Pi\setminus B_{R/4}(R,z)}}|\omg|\bigg)^{2/3}\nonumber\\
			&\lesssim R^{1/3} \bigg\|\frac{\omg}{r}\bigg\|_{L^{\ift}({\Pi})}^{1/3}  R^{-1/3} {\nrm{\omg}_{L^{1}(\Pi)}^{2/3}} {\sim}%=
			\bigg\|\frac{\omg}{r}\bigg\|_{L^{\ift}(\bbR^{3})}^{1/3}  \bigg\|\frac{\omg}{r}\bigg\|_{L^{1}(\bbR^{3})}^{2/3}\label{eq:I1'}
		\end{align}

		\medskip 
		
		\noindent \textbf{Estimate of $ (I_{2}) $: splitting further}. 
		To estimate $(I_{2})$, we introduce the ratio \begin{equation}\label{eq:l}
			\begin{split}
				l:=\frac{1}{ R^{7/4}} \nrm{u}_{L^{2}(\bbR^{3})}^{1/2} \bigg\|\frac{\omg}{r}\bigg\|_{L^{\ift}(\bbR^{3})}^{-1/2}  
			\end{split}
		\end{equation} and further split $(I_{2})$ into
		\begin{equation}\label{eq:I2split}
			%(I_{2})=
			\underbrace{\iint_{\Pi}F^{r}(1-\Phi_{lR})\Phi_{R/2}\omg}_{=:(I_{21})}+\underbrace{\iint_{\Pi}F^{r}\Phi_{lR}\Phi_{R/2}\omg}_{=:(I_{22})}.
		\end{equation}
		Note that $l$ is a non-dimensional parameter. %We consider two cases: $l \ge 1$ and $l < 1$. 
		(When $l \ge 1$, we simply have that $(I_{21}) = 0$.)  {Throughout estimates of $(I_{21})$ and $(I_{22})$, we shall use the estimate
			\begin{equation}\label{eq:FrinI2}
				|F^{r}(R,z,\br{r},\br{z})|\lesssim\frac{|\br{z}-z|}{R^{3/2}\br{r}^{1/2}} \frac{R\br{r}}{D_{{R,z}}^{2}}\leq\frac{1}{R^{1/2}}\frac{\br{r}^{1/2}}{ D_{{R,z}} },
			\end{equation}
			that holds on the region $B_{R/2}(R,z)$.}

		\medskip 
		
		\noindent \textbf{Estimate of $ (I_{22}) $.} Let us first estimate $ (I_{22}) $, which is simpler. We can bound using $\bar{l} := \min\{ l , 1 \}$: \begin{align}
			|(I_{22})|&\leq%\bigg|
			\iint_{{B_{\bar{l}R}(R,z)}}  | F^{r} 
			\omg| %\bigg|
			\lesssim\frac{1}{R^{1/2}} \iint_{{B_{\bar{l}R}(R,z)}}\frac{\br{r}^{1/2}}{D_{{R,z}}} %\Phi_{lR} 
			|\omg|  \frac{\br{r} }{\br{r} }\nonumber\\
			&\lesssim\frac{1}{R^{1/2}} \bigg\|\frac{\omg}{r}\bigg\|_{L^{\ift}({\Pi})}  R^{3/2}  \iint_{{B_{\bar{l}R}(R,z)}}\frac{1%\Phi_{lR}
			}{D_{{R,z}} }  {\lesssim R%\frac{1}{R^{1/2}}
				\bigg\|\frac{\omg}{r}\bigg\|_{L^{\ift}({\Pi})}%R^{3/2}
				\bar{l} R } \nonumber \\
			&\sim  \bar{l}R^{2} \bigg\|\frac{\omg}{r}\bigg\|_{L^{\ift}(\bbR^{3})} \le R^{1/4} \nrm{u}_{L^{2}(\bbR^{3})}^{1/2} \bigg\|\frac{\omg}{r}\bigg\|_{L^{\ift}(\bbR^{3})}^{1/2}. \label{eq:I22'}
		\end{align}  In the last inequality, we have just used $\bar{l}\le l$ and the definition of $l$.

		\medskip 
		
		\noindent \textbf{Estimate of $ (I_{21}) $.} In the estimate of $(I_{21})$, we may assume $l<1$ since otherwise it simply vanishes. %From \eqref{eq:l}, we see that $l<1$ implies \begin{equation}\label{eq:l-imply}
			%	\begin{split}
				%	  \nrm{u}_{L^{2}(\bbR^{3})}^{1/2} \le  R^{7/4}\bigg\|\frac{\omg}{r}\bigg\|_{L^{\ift}(\bbR^{3})}^{1/2}. 
				%	\end{split} \end{equation} 
		Now we use integration by parts to write 
		\begin{equation}\label{eq:I21'ibp}
			\begin{split}
				(I_{21})=&\ \iint_{\Pi}\Big[\ \big[\rd_{\br{z}}\big(F^{r}(1-\Phi_{lR})\Phi_{R/2}\big)\big]u^{r}-\big[\rd_{\br{r}}\big(F^{r}(1-\Phi_{lR})\Phi_{R/2}\big)\big]u^{z}\Big]\\
				=&\ {\underbrace{\iint_{\Pi}(\rd_{\br{z}}F^{r})(1-\Phi_{lR})\Phi_{R/2}u^{r}}_{=:(I_{211})}-\underbrace{\iint_{\Pi}F^{r}(\rd_{\br{z}}\Phi_{lR})\Phi_{R/2}u^{r}}_{=:(I_{212})}+\underbrace{\iint_{\Pi}F^{r}(1-\Phi_{lR})(\rd_{\br{z}}\Phi_{R/2})u^{r}}_{=:(I_{213})}}\\
				&{-\underbrace{\iint_{\Pi}(\rd_{\br{r}}F^{r})(1-\Phi_{lR})\Phi_{R/2}u^{z}}_{=:(I_{214})}+\underbrace{\iint_{\Pi}F^{r}(\rd_{\br{r}}\Phi_{lR})\Phi_{R/2}u^{z}}_{=:(I_{215})}-\underbrace{\iint_{\Pi}F^{r}(1-\Phi_{lR})(\rd_{\br{r}}\Phi_{R/2})u^{z}}_{=:(I_{216})}.}
				%\iint_{\Pi}\Big[\ \big[(\rd_{\br{z}}F^{r})(1-\Phi_{lR})\Phi_{R/2}-F^{r}(\rd_{\br{z}}\Phi_{lR})\Phi_{R/2}+F^{r}(1-\Phi_{lR})(\rd_{\br{z}}\Phi_{R/2})\big]u^{r}\\
				%&-\big[(\rd_{\br{r}}F^{r})(1-\Phi_{lR})\Phi_{R/2}-F^{r}(\rd_{\br{r}}\Phi_{lR})\Phi_{R/2}+F^{r}(1-\Phi_{lR})(\rd_{\br{r}}\Phi_{R/2})\big]u^{z}\Big].
				%\bigg(\bigg)^{1/3}\bigg(\iint_{\Pi}F^{r}(1-\phi_{1/2})\omg\bigg)^{2/3}%\bigg(\iint_{\Pi}F^{r}(1-\phi_{1/2})\omg\bigg)^{1/3}\bigg(\iint_{\Pi}F^{r}(1-\phi_{1/2})\omg\bigg)^{2/3}\\
			\end{split}
		\end{equation}
		Now let us estimate each term. First, we use the estimate
		$$ |\rd_{\br{z}}F^{r}(R,z,\br{r},\br{z})|
		\lesssim\frac{1}{R^{3/2}\br{r}^{1/2}} \bigg[\frac{R\br{r}}{D_{{R,z}}^{2}}+\frac{(\br{z}-z)^{2}}{R\br{r}} \frac{(R\br{r})^{2}}{ D_{{R,z}}^{4}}\bigg]\lesssim\frac{1}{R^{1/2}} \frac{\br{r}^{1/2}}{D_{{R,z}}^{2}} $$
		to get, using $\br{r} \le R$, 
		\begin{align}
			{|(I_{211})|}%\bigg|\iint_{\Pi}(\rd_{\br{z}}F^{r})(1-\Phi_{lR})\Phi_{R/2}u^{r}\bigg|
			&\lesssim\frac{1}{R^{1/2}} \iint_{{B_{R/2}(R,z)\setminus B_{(lR)/2}(R,z)}}\frac{\br{r}^{1/2}}{D_{{R,z}}^{2}} %|1-\Phi_{lR}| |\Phi_{R/2}| 
			|u^{r}|\nonumber\\
			%&\lesssim\iint_{\Pi}\frac{1}{(\br{r}-1)^{2}+\br{z}^{2}} |1-\phi_{k}| |\phi_{1/2}u^{r}|\nonumber\\
			&\lesssim\frac{1}{R^{1/2}} \bigg(\iint_{{B_{R/2}(R,z)\setminus B_{(lR)/2}(R,z)}}\frac{1%|1-\Phi_{lR}|^{2}
			}{ D_{{R,z}}^{4}} \bigg)^{1/2} \bigg(\iint_{{B_{R/2}(R,z)\setminus B_{(lR)/2}(R,z)}}\br{r}|u^{r}|^{2}\bigg)^{1/2}\nonumber\\
			&\lesssim {\frac{1}{R^{1/2}}\frac{1}{lR}\nrm{r^{1/2}u}_{L^{2}(\Pi)}\sim}\frac{1}{l R^{3/2}} \nrm{u}_{L^{2}(\bbR^{3})}.%\lesssim \frac{1}{k} \nrm{\sqrt{r}u}_{L^{2}(\Pi)}.
			\label{eq:rdzFrPhi}
		\end{align}
		Next, using the estimate \eqref{eq:FrinI2} 
		%$$ |F^{r}(R,z,\br{r},\br{z})|\lesssim\frac{|\br{z}-z|}{R^{3/2}\br{r}^{1/2}} \frac{R\br{r}}{D_{{R,z}}^{2}}\leq\frac{1}{R^{1/2}} \frac{\br{r}^{1/2}}{D_{{R,z}}}, $$
		and that   $D_{R,z}\sim lR$ on the support of $\rd_{\br{z}}\Phi_{lR}$, we have
		\begin{align}
			{|(I_{212})|}%\bigg|\iint_{\Pi}F^{r}(\rd_{\br{z}}\Phi_{lR})\Phi_{R/2}u^{r}\bigg|  
			&\lesssim\frac{1}{R^{1/2}} \iint_{{B_{R/2}(R,z)}}\frac{\br{r}^{1/2}}{D_{{R,z}}} \frac{1}{lR} \left|\chi'\bigg(\frac{D_{{R,z}}}{lR}\bigg) \right| \frac{|\br{z}-z|}{D_{{R,z}}} %|\Phi_{R/2}| 
			|u^{r}|\nonumber\\
			&\lesssim\frac{1}{R^{1/2}} \iint_{{B_{R/2}(R,z)\setminus B_{(lR)/2}(R,z)}}\frac{\br{r}^{1/2}}{D_{{R,z}}^{2}} %\left|\chi'\bigg(\frac{D}{lR}\bigg) \right|  
			|u^{r}|
			%&\leq\frac{1}{R^{1/2}} \bigg(\iint_{\Pi}\frac{1}{ D^{4}} \bigg|\chi'\bigg(\frac{D}{lR}\bigg)\bigg|^{2}\bigg)^{1/2} \bigg(\iint_{\Pi}\br{r}|u^{r}|^{2}\bigg)^{1/2}\nonumber\\
			%&\lesssim\bigg(\iint_{\Pi}\frac{1}{[(\br{r}-1)^{2}+\br{z}^{2}]^{4}} |\chi'(2\sqrt{(\br{r}-1)^{2}+\br{z}^{2}})|^{2}\bigg)^{1/2} \bigg(\iint_{\Pi}\br{r}|u^{r}|^{2}\bigg)^{1/2}\nonumber\\
			\lesssim  \frac{1}{l R^{3/2}} \nrm{u}_{L^{2}(\bbR^{3})}.%\frac{1}{k} \nrm{\phi u^{r}}_{L^{2}(\Pi)}\lesssim \frac{1}{k} \nrm{\sqrt{r}u}_{L^{2}(\Pi)}.
			\label{eq:FrrdzPhi}
		\end{align} 
		Similarly, using $D_{R,z}\sim R$ on the support of $\rd_{\br{z}}\Phi_{R/2}$, we get
		\begin{align}
			{|(I_{213})|}%\bigg|\iint_{\Pi} F^{r}(1-\Phi_{lR})(\rd_{\br{z}}\Phi_{R/2})u^{r}\bigg| 
			&\lesssim\frac{1}{R^{1/2}} \iint_{{\Pi\setminus B_{(lR)/2}(R,z)}}\frac{\br{r}^{1/2}}{D_{{R,z}}} \frac{1}{R} %|1-\Phi_{lR}| 
			\chi'\bigg(\frac{2D_{{R,z}}}{R}\bigg) \frac{|\br{z}-z|}{D_{{R,z}}} |u^{r}|\nonumber\\
			&\lesssim\frac{1}{R^{1/2}} \iint_{{B_{R/2}(R,z)\setminus B_{(lR)/2}(R,z)}}\frac{\br{r}^{1/2}}{D_{{R,z}}^{2}} %|1-\Phi_{lR}|% \chi'\bigg(\frac{D}{lR}\bigg)
			|u^{r}|
			%&\lesssim\frac{1}{k} \iint_{\Pi}\frac{1}{\sqrt{(\br{r}-1)^{2}+\br{z}^{2}}} \chi'\bigg(\frac{\sqrt{(\br{r}-1)^{2}+\br{z}^{2}}}{k}\bigg) (|\phi_{1/2}u^{r}|).\nonumber\\
			%&\lesssim\bigg(\iint_{\Pi}\frac{1}{[(\br{r}-1)^{2}+\br{z}^{2}]^{4}} |\chi'(2\sqrt{(\br{r}-1)^{2}+\br{z}^{2}})|^{2}\bigg)^{1/2} \bigg(\iint_{\Pi}\br{r}|u^{r}|^{2}\bigg)^{1/2}\nonumber\\
			\lesssim  \frac{1}{l R^{3/2}} \nrm{u}_{L^{2}(\bbR^{3})}.%\frac{1}{k} \nrm{u^{r}1_{B_{1/2}((1,0))}}_{L^{2}(\Pi)}\lesssim\frac{1}{k} \nrm{\sqrt{r}u}_{L^{2}(\Pi)}.%{\frac{1}{k^{2}}} \nrm{\phi u^{r}}_{L^{2}(\Pi)}\lesssim {\frac{1}{k^{2}}} \nrm{\sqrt{r}u}_{L^{2}(\Pi)}.
			\label{eq:FrPhi}
		\end{align}  
		Now observe that 
		$$ |(\br{r}-R)(\br{r}+R)-(\br{z}-z)^{2}|\lesssim RD_{{R,z}} $$
		holds on the disc $B_{R/2}(R,z)$, which gives \begin{equation}\label{eq:rdrFrRzloc}
			\begin{split}
				|\rd_{\br{r}}F^{r}(R,z,\br{r},\br{z})|&\lesssim\frac{|\br{z}-z|}{(R\br{r})^{5/2}} \bigg[R\br{r} \frac{R\br{r}}{D_{{R,z}}^{2}}+|(\br{r}-R)(\br{r}+R)-(\br{z}-z)^{2}| \frac{(R\br{r})^{2}}{ D_{{R,z}}^{4}}\bigg]\\%\lesssim\frac{\br{r}}{[(\br{r}-1)^{2}+\br{z}^{2}]^{2}}.
				&\lesssim\frac{1}{(R\br{r})^{1/2}} \bigg[\frac{1}{D_{{R,z}}}+\frac{R}{D_{{R,z}}^{2}}\bigg] \lesssim\frac{1}{(R\br{r})^{1/2}} \frac{R}{D_{{R,z}}^{2}}.
			\end{split}
		\end{equation} 
		Using this, we have
		\begin{align}
			{|(I_{214})|}%\bigg|\iint_{\Pi}(\rd_{\br{r}}F^{r})(1-\Phi_{lR})\Phi_{R/2}u^{z}\bigg|
			&\lesssim\frac{1}{R^{1/2}} \iint_{{B_{R/2}(R,z)\setminus B_{(lR)/2}(R,z)}}\frac{1}{\br{r}^{1/2}} \frac{R}{D_{{R,z}}^{2}} %|1-\Phi_{lR}| |\Phi_{R/2}| 
			|u^{z}| \frac{\br{r}^{1/2}}{\br{r}^{1/2}}\nonumber\\
			&\lesssim\frac{1}{R^{1/2}} \iint_{{B_{R/2}(R,z)\setminus B_{(lR)/2}(R,z)}}\frac{\br{r}^{1/2}}{D_{{R,z}}^{2}} %|1-\Phi_{lR}| 
			|u^{z}|
			%&\lesssim\iint_{\Pi}\frac{1}{(\br{r}-1)^{2}+\br{z}^{2}} |1-\phi_{k}| |\phi_{1/2}u^{r}|\nonumber\\
			%&\lesssim\frac{1}{R^{1/2}} \bigg(\iint_{\Pi}\frac{1}{D^{2}}
			%|1-\Phi_{lR}|^{2}\bigg)^{1/2} \bigg(\iint_{\Pi}\br{r}|u^{z}|^{2}\bigg)^{1/2}\nonumber\\
			\lesssim   \frac{1}{l R^{3/2}} \nrm{u}_{L^{2}(\bbR^{3})}.%\lesssim \frac{1}{k} \nrm{\sqrt{r}u}_{L^{2}(\Pi)}.
			\label{eq:rdrFrPhi}
		\end{align}
		Similarly, we can get 
		\begin{equation}\label{eq:FrrdrPhi}
			{|(I_{215})|+|(I_{216})|}%\bigg|\iint_{\Pi}F^{r}(\rd_{\br{r}}\Phi_{lR})\Phi_{R/2}u^{z}\bigg| + 	\bigg|\iint_{\Pi}F^{r}(1-\Phi_{lR})(\rd_{\br{r}}\Phi_{R/2})u^{z}\bigg|
			\lesssim  \frac{1}{l R^{3/2}} \nrm{u}_{L^{2}(\bbR^{3})},
		\end{equation}
		concluding that \begin{equation}\label{eq:I21}
			\begin{split}
				|(I_{21})|&\lesssim \frac{1}{l R^{3/2}} \nrm{u}_{L^{2}(\bbR^{3})} \lesssim R^{1/4} \nrm{u}_{L^{2}(\bbR^{3})}^{1/2} \bigg\|\frac{\omg}{r}\bigg\|_{L^{\ift}(\bbR^{3})}^{1/2},
			\end{split}
		\end{equation} where we have simply used the definition of $l$ from \eqref{eq:l}.

		\medskip 
		
		\noindent \textbf{Conclusion.} Combining \eqref{eq:I21} with \eqref{eq:I22'}, we get that \begin{equation}\label{eq:I2opt}
			\begin{split}
				|(I_{2})|\lesssim R^{1/4} \nrm{u}_{L^{2}(\bbR^{3})}^{1/2} \bigg\|\frac{\omg}{r}\bigg\|_{L^{\ift}(\bbR^{3})}^{1/2}. 
			\end{split}
		\end{equation} Therefore, combining \eqref{eq:I1'} and \eqref{eq:I2opt}, the proof is complete.
	\end{proof}
	
	\medskip 
	
	\subsection{Rate of vortex stretching}\label{sec:main}
	
	In this section, we complete the proof of Theorem \ref{thm:growth-optimal}.

	\begin{proof}[Proof of Theorem \ref{thm:growth-optimal}]  
		We denote the maximal radial support of the solution $\omg(t)$ of \eqref{eq:asEuler} with initial data $\omg_{0}$ as $R(t)$, with  $R_{0}:=R(0)$. Then for any $t\geq0$, using the estimate \eqref{eq:uratRz}, we have
		\begin{equation*}
			\begin{split}
				\left| \frac{d}{dt} R(t) \right| &\le \sup_{z\in\bbR}|u^{r}(t,R,z)| \\
				&\le C\bigg\|\frac{\omg(t,\cdot)}{r}\bigg\|_{L^{\ift}(\bbR^{3})}^{1/3}\bigg\|\frac{\omg(t,\cdot)}{r}\bigg\|_{L^{1}(\bbR^{3})}^{2/3}+CR^{1/4}(t) \nrm{u(t,\cdot)}_{L^{2}(\bbR^{3})}^{1/2}\bigg\|\frac{\omg(t,\cdot)}{r}\bigg\|_{L^{\ift}(\bbR^{3})}^{1/2} \\
				&\le C(c_{1}(\omg_0) + c_{2}(\omg_0) R^{1/4}(t)), 
			\end{split}
		\end{equation*} using conservations of the kinetic energy $\nrm{u}_{L^{2}}$ and the $L^{p}$-norms of $r^{-1}\omg$ with $p\in [1,\ift]$. This gives us the bound 
		\begin{equation}\label{eq:Rt4/3}
			R(t)\lesssim_{\omg_{0}}(1+t)^{4/3},\quad t\geq0.
		\end{equation}
		From this, we obtain
		\begin{equation*}
			\begin{split}
				\nrm{\omg(t,\cdot)}_{L^{\ift}(\bbR^{3})}&\leq\bigg\|\frac{\omg(t,\cdot)}{r}\bigg\|_{L^{\ift}(\bbR^{3})}  R(t)=\bigg\|\frac{\omg_{0}}{r}\bigg\|_{L^{\ift}(\bbR^{3})}  R(t) \lesssim_{\omg_{0}}(1+t)^{4/3},
			\end{split}
		\end{equation*} which finishes the proof.  \end{proof}

	\subsection{Higher dimensions}\label{sec:mainhighd}
	
	{In this section, we present generalizations of the 3D axisymmetric Biot--Savart law from \S\ref{subsec:BS} and the key estimate \eqref{eq:uratRz} from Proposition \ref{prop:urRz} to higher dimensions $d\geq3$. After that, we will prove Theorem \ref{thm:growth-optimalhighd}.}
	
	\subsubsection{Axisymmetric Biot--Savart law in higher dimensions}
	First,  {for any integer $d\geq3$,} we define %introduce 
	the elliptic integral
	\begin{equation}\label{eq:calFhighd}
		\calF_{(d)}(s):=\int_{0}^{\pi}\frac{\cos\alp\sin^{d-3}\alp}{[2(1-\cos\alp)+s]^{d/2-1}}d\alp,\quad s>0,
	\end{equation}
	{which is a high dimensional extension of $\calF=\calF_{(3)}$ from \eqref{eq:calF}.} 
	Then, the axisymmetric Biot--Savart formula can be written in the form {(\cite{GMT2023,Limglobal23})}
	\begin{equation}\label{eq:asBiotSavart_urhighd}
		\begin{split}
			u^{r}(r,z) = \iint_{\Pi}F_{(d)}^{r}(r,z,\br{r},\br{z})\omg(\br{r},\br{z})d\br{z}d\br{r}, \qquad F_{(d)}^{r}(r,z,\br{r},\br{z})=c_{d}  \frac{\br{r}^{d/2-2}(z-\br{z})}{r^{d/2}}\calF_{(d)}'\bigg(\frac{D^{2}}{r\br{r}}\bigg), 
		\end{split}
	\end{equation}
	{for some constant $c_{d}>0$.} 
	We will use the following extension of Lemma \ref{lem:calFest} to higher dimensions.
	\begin{lem}\label{lem:calFdl}
		For any integer $\ell \ge 1$, the $ \ell $-th derivative of $ \calF $ satisfies
		\begin{align} 
			|\calF_{(d)}^{(\ell)}(s)|&\lesssim_{\ell}\min\lbrace s^{-\ell}, s^{-(\ell+d/2)}\rbrace .\label{eq:calFdkest}
		\end{align}
	\end{lem}
	
	{\begin{proof}
			For the case $\ell=1$, we refer the reader to \cite[Lem. 2.2]{CJLglobal22} and its proof. The remaining cases can be shown in a similar way.
	\end{proof}}
	
	\subsubsection{The $R^{d/4-1/2}$ estimate}  {Next, we extend the key estimate \eqref{eq:uratRz} to high dimensional cases. From this estimate, the radial velocity is bounded \textit{on} the maximal support radius of the vorticity.}
	%We prove the key estimate, which bounds the radial velocity \textit{on} the maximal support radius of the vorticity. 
	
	\begin{proposition}\label{prop:urRzhighd}
		Let $\omg$ be compactly supported with maximal radius of the support $R>0$. Then 
		\begin{equation}\label{eq:uratRzhighd}
			\sup_{z\in\bbR}|u^{r}(R,z)|\lesssim\bigg\|\frac{\omg}{r^{d-2}}\bigg\|_{L^{\ift}(\bbR^{d})}^{1/d}\bigg\|\frac{\omg}{r^{d-2}}\bigg\|_{L^{1}(\bbR^{d})}^{1-1/d}+R^{d/4-1/2} \nrm{u}_{L^{2}(\bbR^{d})}^{1/2}\bigg\|\frac{\omg}{r^{d-2}}\bigg\|_{L^{\ift}(\bbR^{d})}^{1/2}.
	\end{equation}	\end{proposition}
	
	\begin{proof}
		We use the rescaled cutoff function \eqref{eq:Phi} once more to extend the proof of Proposition \ref{prop:urRz} to higher dimensions. First, we split $u^{r}(R,z)$ into
		\begin{equation}\label{eq:urRzsplithighd}
			\begin{split}
				&\ \underbrace{\iint_{\Pi}F_{(d)}^{r}(R,z,\br{r},\br{z})\big(1-\Phi_{R/2}(\br{r},\br{z})\big)\omg(\br{r},\br{z})d\br{z}d\br{r}}_{=:(I_{1}')} +\underbrace{\iint_{\Pi}F_{(d)}^{r}(R,z,\br{r},\br{z})\Phi_{R/2}(\br{r},\br{z})\omg(\br{r},\br{z})d\br{z}d\br{r}}_{=:(I_{2}')},
			\end{split}
		\end{equation}
		and abbreviate integrals $ (I_{1}') $ and $ (I_{2}') $ as
		$$ (I_{1}')=\iint_{\Pi}F_{(d)}^{r}(1-\Phi_{R/2})\omg,\quad (I_{2}')=\iint_{\Pi}F_{(d)}^{r}\Phi_{R/2}\omg. $$
		\medskip 
		
		\noindent \textbf{Estimate of $ (I_{1}') $}.  {Recall from the 3D case that} %In 
		the term $ (I_{1}') $ 
		%, the integrand 
		is supported on the region $ \Pi\setminus B_{R/4}(R,z) $.  {Then using the estimate}%To estimate this term, we use the estimate
		\begin{equation}\label{eq:FrRznonlochighd}
			\begin{split}
				|F_{(d)}^{r}(R,z,\br{r},\br{z})|&\lesssim\frac{\br{r}^{d/2-2}|\br{z}-z|}{R^{d/2}} \frac{(R\br{r})^{d/2+1}}{D_{{R,z}}^{d+2}}\leq\frac{R\br{r}^{d-1}}{ D_{{R,z}}^{d+1} },
			\end{split}
		\end{equation}
		that holds on the region $\Pi\setminus B_{R/4}(R,z)$, and 
		%Then we have, 
		using $\br{r} \le  R  \le 4 D_{{R,z}}, $ %\footnote{This is the only place where we really need that $R$ is the maximal support radius of $\omg$.}
		{we can estimate this term as}
		\begin{align}
			|(I_{1}')|&\lesssim\iint_{{\Pi\setminus B_{R/4}(R,z)}}\frac{R\br{r}^{d-1} %|1-\Phi_{R/2}|
			}{D_{{R,z}}^{d+1}}  |\omg|^{1-1/d} |\omg|^{1/d} \frac{\br{r}^{1-2/d}}{\br{r}^{1-2/d}}\nonumber\\
			&\lesssim R^{1-2/d} \bigg\|\frac{\omg}{r^{d-2}}\bigg\|_{L^{\ift}({\Pi})}^{1/d} \iint_{{\Pi\setminus B_{R/4}(R,z)}}\frac{1%|1-\Phi_{R/2}|
			}{D_{{R,z}}}  |\omg|^{1-1/d}\nonumber\\
			&\leq R^{1-2/d} \bigg\|\frac{\omg}{r^{d-2}}\bigg\|_{L^{\ift}({\Pi})}^{1/d} \bigg(\iint_{{\Pi\setminus B_{R/4}(R,z)}}\frac{1%|1-\Phi_{R/2}|^{3}
			}{D_{{R,z}}^{d}} \bigg)^{1/d} \bigg(\iint_{{\Pi\setminus B_{R/4}(R,z)}}|\omg|\bigg)^{1-1/d}\nonumber\\
			&\lesssim R^{1-2/d} \bigg\|\frac{\omg}{r^{d-2}}\bigg\|_{L^{\ift}({\Pi})}^{1/d}  R^{-(1-2/d)} {\nrm{\omg}%\bigg\|\frac{\omg}{r}\bigg\|
				_{L^{1}(\Pi)}^{1-1/d}} {\sim}%=
			\bigg\|\frac{\omg}{r^{d-2}}\bigg\|_{L^{\ift}(\bbR^{d})}^{1/d}  \bigg\|\frac{\omg}{r^{d-2}}\bigg\|_{L^{1}(\bbR^{d})}^{1-1/d}.\label{eq:I1''}
		\end{align}
		\medskip 
		
		\noindent \textbf{Estimate of $ (I_{2}') $: splitting further}. 
		{As we did in the case $d=3$, we take the ratio} %To estimate $(I_{2}')$, we introduce the ratio
		\begin{equation}\label{eq:lhighd}
			\begin{split}
				l:=R^{-(3d)/4+1/2} \nrm{u}_{L^{2}(\bbR^{d})}^{1/2} \bigg\|\frac{\omg}{r^{d-2}}\bigg\|_{L^{\ift}(\bbR^{d})}^{-1/2},  
			\end{split}
		\end{equation} and further split the term $(I_{2}')$ as%into
		\begin{equation*}
			%(I_{2}')=
			\underbrace{\iint_{\Pi}F_{(d)}^{r}(1-\Phi_{lR})\Phi_{R/2}\omg}_{=:(I_{21}')}+\underbrace{\iint_{\Pi}F_{(d)}^{r}\Phi_{lR}\Phi_{R/2}\omg}_{=:(I_{22}')}.
		\end{equation*}
		%Note that $l$ is a non-dimensional parameter. %We consider two cases: $l \ge 1$ and $l < 1$. 
		{(We remind the reader that we simply have $(I_{21}') = 0$ for the case $l\geq1$.) We often use the estimate \begin{equation}\label{eq:FrinI2highd}
				|F_{(d)}^{r}(R,z,\br{r},\br{z})|\lesssim\frac{\br{r}^{d/2-2}|\br{z}-z|}{R^{d/2}} \frac{R\br{r}}{D_{{R,z}}^{2}}\leq\frac{\br{r}^{d/2-1}}{R^{d/2-1}} \frac{1}{D_{{R,z}}},
			\end{equation}
			which holds on the region $B_{R/2}(R,z)$, during estimates of $(I_{21}')$ and $(I_{22}')$.}%(When $l \ge 1$, we simply have that $(I_{21}) = 0$.)
		
		\medskip 
		
		\noindent \textbf{Estimate of $ (I_{22}') $.}  {The estimate of $(I_{22}')$ is easily done by using $\bar{l} := \min\{ l , 1 \}$:}%Let us first estimate $ (I_{22}') $, which is simpler. We can bound using $\bar{l} := \min\{ l , 1 \}$ 
		\begin{align}
			|(I_{22}')|&\leq%\bigg|
			\iint_{{B_{\bar{l}R}(R,z)}}  | F_{(d)}^{r} 
			\omg| %\bigg|
			\lesssim\frac{1}{R^{d/2-1}} \iint_{{B_{\bar{l}R}(R,z)}}\frac{\br{r}^{d/2-1}}{D_{{R,z}}} %\Phi_{lR} 
			|\omg|  \frac{\br{r}^{d-2} }{\br{r}^{d-2} }\nonumber\\
			&\lesssim\frac{1}{R^{d/2-1}} \bigg\|\frac{\omg}{r^{d-2}}\bigg\|_{L^{\ift}({\Pi})}  R^{(3d)/2-3}  \iint_{{B_{\bar{l}R}(R,z)}}\frac{1%\Phi_{lR}
			}{D_{{R,z}} }  \nonumber\\
			&{\lesssim\frac{1}{R^{d/2-1}}\bigg\|\frac{\omg}{r^{d-2}}\bigg\|_{L^{\ift}({\Pi})}R^{(3d)/2-3} \bar{l} R } \nonumber \\
			&\sim  \bar{l}R^{d-1} \bigg\|\frac{\omg}{r^{d-2}}\bigg\|_{L^{\ift}(\bbR^{d})} \le R^{d/4-1/2} \nrm{u}_{L^{2}(\bbR^{d})}^{1/2} \bigg\|\frac{\omg}{r^{d-2}}\bigg\|_{L^{\ift}(\bbR^{d})}^{1/2}, \label{eq:I22''}
		\end{align}   {where we simply used $\bar{l}\le l$ and the definition of $l$ in the last inequality.}%In the last inequality, we have just used $\bar{l}\le l$ and the definition of $l$. 
		
		\medskip 
		
		\noindent \textbf{Estimate of $ (I_{21}') $.}  {Since $(I_{21}')$ becomes zero when $l\geq1$, we only consider the case $l<1$. }%In the estimate of $(I_{21}')$, we may assume $l<1$ since otherwise it simply vanishes. 
		%From \eqref{eq:l}, we see that $l<1$ implies \begin{equation}\label{eq:l-imply}
			%	\begin{split}
				%	  \nrm{u}_{L^{2}(\bbR^{3})}^{1/2} \le  R^{7/4}\bigg\|\frac{\omg}{r}\bigg\|_{L^{\ift}(\bbR^{3})}^{1/2}. 
				%	\end{split} \end{equation} 
		We use integration by parts  {once more} to write 
		\begin{equation}\label{eq:I21''ibp}
			\begin{split}
				(I_{21}')=&\ \iint_{\Pi}\Big[\ \big[\rd_{\br{z}}\big(F_{(d)}^{r}(1-\Phi_{lR})\Phi_{R/2}\big)\big]u^{r}-\big[\rd_{\br{r}}\big(F_{(d)}^{r}(1-\Phi_{lR})\Phi_{R/2}\big)\big]u^{z}\Big]\\
				=&\ {\underbrace{\iint_{\Pi}(\rd_{\br{z}}F_{(d)}^{r})(1-\Phi_{lR})\Phi_{R/2}u^{r}}_{=:(I_{211}')}-\underbrace{\iint_{\Pi}F_{(d)}^{r}(\rd_{\br{z}}\Phi_{lR})\Phi_{R/2}u^{r}}_{=:(I_{212}')}+\underbrace{\iint_{\Pi}F_{(d)}^{r}(1-\Phi_{lR})(\rd_{\br{z}}\Phi_{R/2})u^{r}}_{=:(I_{213}')}}\\
				&{-\underbrace{\iint_{\Pi}(\rd_{\br{r}}F_{(d)}^{r})(1-\Phi_{lR})\Phi_{R/2}u^{z}}_{=:(I_{214}')}+\underbrace{\iint_{\Pi}F_{(d)}^{r}(\rd_{\br{r}}\Phi_{lR})\Phi_{R/2}u^{z}}_{=:(I_{215}')}-\underbrace{\iint_{\Pi}F_{(d)}^{r}(1-\Phi_{lR})(\rd_{\br{r}}\Phi_{R/2})u^{z}}_{=:(I_{216}')}.}
				%\iint_{\Pi}\Big[\ \big[(\rd_{\br{z}}F^{r})(1-\Phi_{lR})\Phi_{R/2}-F^{r}(\rd_{\br{z}}\Phi_{lR})\Phi_{R/2}+F^{r}(1-\Phi_{lR})(\rd_{\br{z}}\Phi_{R/2})\big]u^{r}\\
				%&-\big[(\rd_{\br{r}}F^{r})(1-\Phi_{lR})\Phi_{R/2}-F^{r}(\rd_{\br{r}}\Phi_{lR})\Phi_{R/2}+F^{r}(1-\Phi_{lR})(\rd_{\br{r}}\Phi_{R/2})\big]u^{z}\Big].
				%\bigg(\bigg)^{1/3}\bigg(\iint_{\Pi}F^{r}(1-\phi_{1/2})\omg\bigg)^{2/3}%\bigg(\iint_{\Pi}F^{r}(1-\phi_{1/2})\omg\bigg)^{1/3}\bigg(\iint_{\Pi}F^{r}(1-\phi_{1/2})\omg\bigg)^{2/3}\\
			\end{split}
		\end{equation}
		{To estimate $(I_{211}')$, }%Now let us estimate each term. First, 
		we use the estimate
		$$ |\rd_{\br{z}}F_{(d)}^{r}(R,z,\br{r},\br{z})|
		\lesssim\frac{\br{r}^{d/2-2}}{R^{d/2}} \bigg[\frac{R\br{r}}{D_{{R,z}}^{2}}+\frac{(\br{z}-z)^{2}}{R\br{r}} \frac{(R\br{r})^{2}}{ D_{{R,z}}^{4}}\bigg]\lesssim\frac{\br{r}^{d/2-1}}{R^{d/2-1}} \frac{1}{D_{{R,z}}^{2}}. $$
		%to get, 
		Using  {this and} $\br{r} \le R$,  {we have}
		\begin{align}
			{|(I_{211}')|}%\bigg|\iint_{\Pi}(\rd_{\br{z}}F^{r})(1-\Phi_{lR})\Phi_{R/2}u^{r}\bigg|
			&\lesssim\frac{1}{R^{d/2-1}} \iint_{{B_{R/2}(R,z)\setminus B_{(lR)/2}(R,z)}}\frac{\br{r}^{d/2-1}}{D_{{R,z}}^{2}} %|1-\Phi_{lR}| |\Phi_{R/2}| 
			|u^{r}|\nonumber\\
			%&\lesssim\iint_{\Pi}\frac{1}{(\br{r}-1)^{2}+\br{z}^{2}} |1-\phi_{k}| |\phi_{1/2}u^{r}|\nonumber\\
			&\lesssim\frac{1}{R^{d/2-1}} \bigg(\iint_{{B_{R/2}(R,z)\setminus B_{(lR)/2}(R,z)}}\frac{1%|1-\Phi_{lR}|^{2}
			}{ D_{{R,z}}^{4}} \bigg)^{1/2} \bigg(\iint_{{B_{R/2}(R,z)\setminus B_{(lR)/2}(R,z)}}\br{r}^{d-2}|u^{r}|^{2}\bigg)^{1/2}\nonumber\\
			&\lesssim {\frac{1}{R^{d/2-1}}\frac{1}{lR}\nrm{r^{d/2-1}u}_{L^{2}(\Pi)}\sim}\frac{1}{l R^{d/2}} \nrm{u}_{L^{2}(\bbR^{d})}.%\lesssim \frac{1}{k} \nrm{\sqrt{r}u}_{L^{2}(\Pi)}.
			\label{eq:rdzFrPhihighd}
		\end{align}
		Next,  {we use} %using 
		the estimate \eqref{eq:FrinI2highd} and  {the relation} %that   
		$D_{R,z}\sim lR$,  {which holds} on the support of $\rd_{\br{z}}\Phi_{lR}$,  {to get}%, we have
		\begin{align}
			{|(I_{212}')|}%\bigg|\iint_{\Pi}F^{r}(\rd_{\br{z}}\Phi_{lR})\Phi_{R/2}u^{r}\bigg|  
			&\lesssim\frac{1}{R^{d/2-1}} \iint_{{B_{R/2}(R,z)}}\frac{\br{r}^{d/2-1}}{D_{{R,z}}} \frac{1}{lR} \left|\chi'\bigg(\frac{D_{{R,z}}}{lR}\bigg) \right| \frac{|\br{z}-z|}{D_{{R,z}}} %|\Phi_{R/2}| 
			|u^{r}|\nonumber\\
			&\lesssim\frac{1}{R^{d/2-1}} \iint_{{B_{R/2}(R,z)\setminus B_{(lR)/2}(R,z)}}\frac{\br{r}^{d/2-1}}{D_{{R,z}}^{2}} %\left|\chi'\bigg(\frac{D}{lR}\bigg) \right|  
			|u^{r}|
			%&\leq\frac{1}{R^{1/2}} \bigg(\iint_{\Pi}\frac{1}{ D^{4}} \bigg|\chi'\bigg(\frac{D}{lR}\bigg)\bigg|^{2}\bigg)^{1/2} \bigg(\iint_{\Pi}\br{r}|u^{r}|^{2}\bigg)^{1/2}\nonumber\\
			%&\lesssim\bigg(\iint_{\Pi}\frac{1}{[(\br{r}-1)^{2}+\br{z}^{2}]^{4}} |\chi'(2\sqrt{(\br{r}-1)^{2}+\br{z}^{2}})|^{2}\bigg)^{1/2} \bigg(\iint_{\Pi}\br{r}|u^{r}|^{2}\bigg)^{1/2}\nonumber\\
			\lesssim  \frac{1}{l R^{d/2}} \nrm{u}_{L^{2}(\bbR^{d})}.%\frac{1}{k} \nrm{\phi u^{r}}_{L^{2}(\Pi)}\lesssim \frac{1}{k} \nrm{\sqrt{r}u}_{L^{2}(\Pi)}.
			\label{eq:FrrdzPhihighd}
		\end{align} 
		{Likewise, we use} %Similarly, using 
		$D_{R,z}\sim R$ on the support of $\rd_{\br{z}}\Phi_{R/2}$,  {which gives us}
		\begin{align}
			{|(I_{213}')|}%\bigg|\iint_{\Pi} F^{r}(1-\Phi_{lR})(\rd_{\br{z}}\Phi_{R/2})u^{r}\bigg| 
			&\lesssim\frac{1}{R^{d/2-1}} \iint_{{\Pi\setminus B_{(lR)/2}(R,z)}}\frac{\br{r}^{d/2-1}}{D_{{R,z}}} \frac{1}{R} %|1-\Phi_{lR}| 
			\chi'\bigg(\frac{2D_{{R,z}}}{R}\bigg) \frac{|\br{z}-z|}{D_{{R,z}}} |u^{r}|\nonumber\\
			&\lesssim\frac{1}{R^{d/2-1}} \iint_{{B_{R/2}(R,z)\setminus B_{(lR)/2}(R,z)}}\frac{\br{r}^{d/2-1}}{D_{{R,z}}^{2}} %|1-\Phi_{lR}|% \chi'\bigg(\frac{D}{lR}\bigg)
			|u^{r}|
			%&\lesssim\frac{1}{k} \iint_{\Pi}\frac{1}{\sqrt{(\br{r}-1)^{2}+\br{z}^{2}}} \chi'\bigg(\frac{\sqrt{(\br{r}-1)^{2}+\br{z}^{2}}}{k}\bigg) (|\phi_{1/2}u^{r}|).\nonumber\\
			%&\lesssim\bigg(\iint_{\Pi}\frac{1}{[(\br{r}-1)^{2}+\br{z}^{2}]^{4}} |\chi'(2\sqrt{(\br{r}-1)^{2}+\br{z}^{2}})|^{2}\bigg)^{1/2} \bigg(\iint_{\Pi}\br{r}|u^{r}|^{2}\bigg)^{1/2}\nonumber\\
			\lesssim  \frac{1}{l R^{d/2}} \nrm{u}_{L^{2}(\bbR^{d})}.%\frac{1}{k} \nrm{u^{r}1_{B_{1/2}((1,0))}}_{L^{2}(\Pi)}\lesssim\frac{1}{k} \nrm{\sqrt{r}u}_{L^{2}(\Pi)}.%{\frac{1}{k^{2}}} \nrm{\phi u^{r}}_{L^{2}(\Pi)}\lesssim {\frac{1}{k^{2}}} \nrm{\sqrt{r}u}_{L^{2}(\Pi)}.
			\label{eq:FrPhihighd}
		\end{align}  
		Now  {to estimate $(I_{214}')$, we use the inequality}%observe that 
		$$ |(\br{r}-R)(\br{r}+R)-(\br{z}-z)^{2}|\lesssim RD_{{R,z}}, $$
		{which} holds on the disc $B_{R/2}(R,z)$,  {to get}%which gives
		\begin{equation}\label{eq:rdrFrRzlochighd}
			\begin{split}
				|\rd_{\br{r}}F_{(d)}^{r}(R,z,\br{r},\br{z})|&\lesssim\frac{\br{r}^{d/2-4}|\br{z}-z|}{R^{d/2+1}} \bigg[R\br{r} \frac{R\br{r}}{D_{{R,z}}^{2}}+|(\br{r}-R)(\br{r}+R)-(\br{z}-z)^{2}| \frac{(R\br{r})^{2}}{ D_{{R,z}}^{4}}\bigg]\\%\lesssim\frac{\br{r}}{[(\br{r}-1)^{2}+\br{z}^{2}]^{2}}.
				&\lesssim\frac{\br{r}^{d/2-2}}{R^{d/2-1}} \bigg[\frac{1}{D_{{R,z}}}+\frac{R}{D_{{R,z}}^{2}}\bigg] \lesssim\frac{\br{r}^{d/2-2}}{R^{d/2-1}} \frac{R}{D_{{R,z}}^{2}}.
			\end{split}
		\end{equation} 
		{This gives us}%Using this, we have
		\begin{align}
			{|(I_{214}')|}%\bigg|\iint_{\Pi}(\rd_{\br{r}}F^{r})(1-\Phi_{lR})\Phi_{R/2}u^{z}\bigg|
			&\lesssim\frac{1}{R^{d/2-1}} \iint_{{B_{R/2}(R,z)\setminus B_{(lR)/2}(R,z)}}\br{r}^{d/2-2} \frac{R}{D_{{R,z}}^{2}} %|1-\Phi_{lR}| |\Phi_{R/2}| 
			|u^{z}| %\frac{\br{r}^{1/2}}{\br{r}^{1/2}}
			\nonumber\\
			&\lesssim\frac{1}{R^{d/2-1}} \iint_{{B_{R/2}(R,z)\setminus B_{(lR)/2}(R,z)}}\frac{\br{r}^{d/2-1}}{D_{{R,z}}^{2}} %|1-\Phi_{lR}| 
			|u^{z}|
			%&\lesssim\iint_{\Pi}\frac{1}{(\br{r}-1)^{2}+\br{z}^{2}} |1-\phi_{k}| |\phi_{1/2}u^{r}|\nonumber\\
			%&\lesssim\frac{1}{R^{1/2}} \bigg(\iint_{\Pi}\frac{1}{D^{2}}
			%|1-\Phi_{lR}|^{2}\bigg)^{1/2} \bigg(\iint_{\Pi}\br{r}|u^{z}|^{2}\bigg)^{1/2}\nonumber\\
			\lesssim   \frac{1}{l R^{d/2}} \nrm{u}_{L^{2}(\bbR^{d})}.%\lesssim \frac{1}{k} \nrm{\sqrt{r}u}_{L^{2}(\Pi)}.
			\label{eq:rdrFrPhihighd}
		\end{align}
		Similarly, we  {have}%can get 
		\begin{equation}\label{eq:FrrdrPhihighd}
			{|(I_{215}')|+|(I_{216}')|}%\bigg|\iint_{\Pi}F^{r}(\rd_{\br{r}}\Phi_{lR})\Phi_{R/2}u^{z}\bigg| + 	\bigg|\iint_{\Pi}F^{r}(1-\Phi_{lR})(\rd_{\br{r}}\Phi_{R/2})u^{z}\bigg|
			\lesssim  \frac{1}{l R^{d/2}} \nrm{u}_{L^{2}(\bbR^{d})}.
		\end{equation}
		{As a result, using the definition of $l$ from \eqref{eq:lhighd}, we obtain}%concluding that 
		\begin{equation}\label{eq:I21''}
			\begin{split}
				|(I_{21}')|&\lesssim \frac{1}{l R^{d/2}} \nrm{u}_{L^{2}(\bbR^{d})} \lesssim R^{d/4-1/2} \nrm{u}_{L^{2}(\bbR^{d})}^{1/2} \bigg\|\frac{\omg}{r^{d-2}}\bigg\|_{L^{\ift}(\bbR^{d})}^{1/2}.
			\end{split}
		\end{equation} %where we have simply used the definition of $l$ from \eqref{eq:lhighd}. 
	\end{proof}

	\subsubsection{Rate of vortex stretching}
	
	We may now complete the proof of Theorem \ref{thm:growth-optimalhighd}.

	\begin{proof}[Proof of Theorem \ref{thm:growth-optimalhighd}]  
		{We use the same notation of $R(t)$ from the proof of Theorem \ref{thm:growth-optimal}.} %We denote the maximal radial support of the solution $\omg(t)$ of \eqref{eq:asEulerhighd} with initial data $\omg_{0}$ as $R(t)$, with  $R_{0}:=R(0)$. 
		Then for any $t\geq0$,  {we use the estimate \eqref{eq:uratRzhighd} to obtain}
		%using the estimate \eqref{eq:uratRzhighd}, we have
		\begin{equation*}
			\begin{split}
				\left| \frac{d}{dt} R(t) \right| &\le  C\bigg\|\frac{\omg(t,\cdot)}{r^{d-2}}\bigg\|_{L^{\ift}(\bbR^{d})}^{1/d}\bigg\|\frac{\omg(t,\cdot)}{r^{d-2}}\bigg\|_{L^{1}(\bbR^{d})}^{1-1/d}+CR^{d/4-1/2}(t) \nrm{u(t,\cdot)}_{L^{2}(\bbR^{d})}^{1/2}\bigg\|\frac{\omg(t,\cdot)}{r^{d-2}}\bigg\|_{L^{\ift}(\bbR^{d})}^{1/2} \\
				&\le C(c_{1}(\omg_0) + c_{2}(\omg_0) R^{d/4-1/2}(t)), 
			\end{split}
		\end{equation*}  {where we used} %using 
		conservations of the kinetic energy $\nrm{u}_{L^{2}}$ and the $L^{p}$-norms of $r^{-(d-2)}\omg$.  {From this, we have}%This gives us 
		\begin{equation}\label{eq:Rt4/3highd}
			R(t)\lesssim_{\omg_{0}}\begin{cases}
				A_{1}(d) (1+t)^{4/(6-d)},&\quad d=3,4,5,\\
				A_{2}e^{A_{3}t},&\quad d=6,
			\end{cases}%(1+t)^{4/(6-d)},
			\quad t\geq0,
		\end{equation}
		{which gives us}%From this, we obtain
		\begin{equation*}
			\begin{split}
				\nrm{\omg(t,\cdot)}_{L^{\ift}(\bbR^{d})}&\leq\bigg\|\frac{\omg(t,\cdot)}{r^{d-2}}\bigg\|_{L^{\ift}(\bbR^{d})}  R(t)=\bigg\|\frac{\omg_{0}}{r^{d-2}}\bigg\|_{L^{\ift}(\bbR^{d})}  R(t) \lesssim_{\omg_{0}}\begin{cases}
					A_{1} (1+t)^{4/(6-d)},&\, d=3,4,5,\\
					A_{2}e^{A_{3}t},&\, d=6.
				\end{cases}%(1+t)^{4/(6-d)},
			\end{split}
		\end{equation*}  {This} %which 
		finishes the proof.  \end{proof} 
	
	\appendix
	
	\section{Global estimate for the radial velocity}\label{sec:growth3/2} 
	{In this section, we provide {proofs} %the proof 
		of the inequality \eqref{eq:newest} {and Theorem \ref{thm:t3/2}} 
		which we introduced in \S\ref{sec:intro}.}
	
	\begin{prop}\label{prop:estimate}
		Let $\omg$ satisfy  
		$r^{-1}\omg\in L^{\ift}(\bbR^{3})$ and $r\omg%\nb u
		\in L^{1}(\bbR^{3})$. Furthermore, assume that the corresponding velocity $u = (u^{r}, u^{z})$ belongs to $L^{2}(\bbR^{3})$. Then, $u^{r}$ is uniformly bounded in space, with 
		\begin{equation}\label{eq:uiftnewest}
			\nrm{u^{r}}_{L^{\ift}(\bbR^{3})} \le C\nrm{u}_{L^{2}(\bbR^{3})}^{1/3}\bigg\|\frac{\omg}{r}\bigg\|_{L^{\ift}(\bbR^{3})}^{1/2}\nrm{r\omg}_{L^{1}(\bbR^{3})}^{1/6}
		\end{equation} for a universal constant $C>0$. 
	\end{prop}
	
	\begin{rem}\label{rem:1,0}
		Due to invariances of the Biot--Savart formula with respect to the following scaling and translation in $ z $
		\begin{equation*}
			u(r,z)\mapsto u(\lmb r,\lmb z+z_{0})\quad\text{and}\quad \omg(r,z)\mapsto \lmb\omg(\lmb r, \lmb z+z_{0})\quad\text{for any}\quad(\lmb,z_{0})\in\bbR_{>0}\times\bbR,
		\end{equation*}
		to prove Proposition \ref{prop:estimate}, it suffices to show
		\begin{equation}\label{eq:1,0}
			|u^{r}(1,0)|\lesssim\nrm{u}_{L^{2}(\bbR^{3})}^{1/3}\bigg\|\frac{\omg}{r}\bigg\|_{L^{\ift}(\bbR^{3})}^{1/2}\nrm{r\omg}_{L^{1}(\bbR^{3})}^{1/6}.
		\end{equation}
		\begin{comment}
			Or, using the following identities
			$$ \nrm{u}_{L^{2}(\bbR^{3})}\sim\nrm{r^{1/2}u}_{L^{2}(\Pi)},\quad\bigg\|\frac{\omg}{r}\bigg\|_{L^{\ift}(\bbR^{3})}\sim\bigg\|\frac{\omg}{r}\bigg\|_{L^{\ift}(\Pi)},\quad\nrm{r\omg}_{L^{1}(\bbR^{3})}\sim\nrm{r^{2}\omg}_{L^{1}(\Pi)}, $$
			we can write \eqref{eq:1,0} as
			\begin{equation}\label{eq:1,0id}
				|u^{r}(1,0)|\lesssim\nrm{r^{1/2}u}_{L^{2}(\Pi)}^{1/3}\bigg\|\frac{\omg}{r}\bigg\|_{L^{\ift}(\Pi)}^{1/2}\nrm{r^{2}\omg}_{L^{1}(\Pi)}^{1/6}.
			\end{equation}
		\end{comment}
	\end{rem}

	\subsection{Proof of Proposition \ref{prop:estimate} 
	}\label{ssec:ur}	
	
	Our goal is to prove \eqref{eq:1,0} from Remark \ref{rem:1,0}. To do this, %let us 
	%introduce 
	{we define the rescaled cutoff functions}%some notations. 
	%\medskip
	%Also, let $ k>0 $ and let us denote the disc on the half plane $ \Pi $ centered at the point $ (1,0) $ with radius $ k $ as
	%$$ B_{k}(1,0):= \lbrace(r,z)\in\Pi : D(r,z)<k\rbrace ,$$ where $D(r,z) =  |(r,z)-(1,0)| = ( (r-1)^2 + z^2 )^{1/2}$. 
	%In addition, 
	%We let $k>0$ and define $ \phi_{k}\in C_{c}^{\ift}(\Pi;[0,1]) $ as
	\begin{equation}\label{eq:phi}
		\phi_{k}(\br{r},\br{z}):=\chi\bigg(\frac{D(1,0,\br{r},\br{z})}{k}\bigg) ,\quad k>0.
	\end{equation}
	%That is, $ \phi_{k} $ is supported on the disc $ B_{k}(1,0) $. 
	
	\begin{proof}[Proof of Proposition \ref{prop:estimate}]  
		Using the function $ \phi_{1/2} $, where we took $ k=1/2 $, we split the term $ u^{r}(1,0) $ into
		\begin{equation}\label{eq:ursplit}
			\begin{split}
				%u^{r}(1,0) =
				&\ \underbrace{\iint_{\Pi}F^{r}(1,0,\br{r},\br{z})\big(1-\phi_{1/2}(\br{r},\br{z})\big)\omg(\br{r},\br{z})d\br{z}d\br{r}}_{=:(I_{1}'')} + \underbrace{\iint_{\Pi}F^{r}(1,0,\br{r},\br{z})\phi_{1/2}(\br{r},\br{z})\omg(\br{r},\br{z})d\br{z}d\br{r}}_{=:(I_{2}'')}.
			\end{split}
		\end{equation}
		For notational convenience, let us  {simply write} %abbreviate integrals 
		$ (I_{1}'') $ and $ (I_{2}'') $ as
		$$ (I_{1}'')=\iint_{\Pi}F^{r}(1-\phi_{1/2})\omg,\quad (I_{2}'')=\iint_{\Pi}F^{r}\phi_{1/2}\omg. $$
		{In addition, we simplify $D(1,0,\br{r},\br{z})$ as $D_{1,0}$.}

		\medskip
		
		\noindent \textbf{Estimate of $(I_{1}'') $.} %Note that the integrand of $ (I_{1}'') $ is supported on the region . 
		We are going to estimate $(I_1'')$ {, which is supported on $ \Pi\setminus B_{1/4}(1,0)$,} in two different ways, and combine the inequalities by splitting $ (I_{1}'')=(I_{1}'')^{1/3}(I_{1}'')^{2/3}$. The first estimate follows  {by using the estimate}%from using Lemma \ref{lem:calFest} with the exponent $ -5/2 $, which gives 
		\begin{equation}\label{eq:Frnonloc}
			\left| F^{r}(1,0,\br{r},\br{z}) \right| =\frac{1}{\pi} \left| \frac{\br{z}}{\br{r}^{1/2}} \calF'\bigg(\frac{D_{1,0}^{{2}}%(\br{r}-1)^{2}+\br{z}^{2}
			}{\br{r}}\bigg) \right| \lesssim\frac{|\br{z}|}{\br{r}^{1/2}} \frac{\br{r}^{5/2}}{{D_{1,0}^{5}}%[(\br{r}-1)^{2}+\br{z}^{2}]^{5/2}
			}\leq\frac{\br{r}^{2}}{{D_{1,0}^{4}}%[(\br{r}-1)^{2}+\br{z}^{2}]^{2}
			}.
		\end{equation} 
		Using this, we  {have}%get
		\begin{equation*}
			\begin{split}
				|(I_{1}'')|&\lesssim\iint_{{\Pi\setminus B_{1/4}(1,0)}}\frac{\br{r}^{2}}{{D_{1,0}^{4}}%[(\br{r}-1)^{2}+\br{z}^{2}]^{2}
				} %|1-\phi_{1/2}| 
				|\omg|^{3/4} |\omg|^{1/4} \frac{\br{r}^{3/4}}{\br{r}^{3/4}}\leq\bigg\|\frac{\omg}{r}\bigg\|_{L^{\ift}(\Pi)}^{3/4} \iint_{{\Pi\setminus B_{1/4}(1,0)%\lbrace D_{1,0}\geq1/4\rbrace
					}
				}\frac{\br{r}^{9/4}}{{D_{1,0}^{4}}%[(\br{r}-1)^{2}+\br{z}^{2}]^{2}
				} %|1-\phi_{1/2}| 
				(\br{r}^{1/2}|\omg|^{1/4}).
			\end{split}
		\end{equation*}
		{Here, we use the  {inequality $\br{r}\lesssim D_{1,0}$,} %following relation:
			\begin{comment}\label{eq:rD1,0}
				%\Pi\setminus B_{1/4}(1,0)\subset \lbrace (\br{r},\br{z})\in \Pi : \br{r}\leq 10D_{1,0}\rbrace,
			\end{comment}
			%which allows us to use the inequality $ \br{r}\lesssim D_{1,0}, $
			{which holds} on the region $\Pi\setminus B_{1/4}(1,0)$,} %We use this 
		to finish the first estimate of $ (I_{1}'') $:
		\begin{align}
			|(I_{1}'')|&\lesssim\bigg\|\frac{\omg}{r}\bigg\|_{L^{\ift}(\Pi)}^{3/4} \iint_{{\Pi\setminus B_{1/4}(1,0)%\lbrace D_{1,0}\geq1/4\rbrace
				}%\Pi
			}\frac{\br{r}^{9/4}}{{D_{1,0}^{4}}%[(\br{r}-1)^{2}+\br{z}^{2}]^{2}
			} %|1-\phi_{1/2}| 
			(\br{r}^{1/2}|\omg|^{1/4})\lesssim\bigg\|\frac{\omg}{r}\bigg\|_{L^{\ift}(\Pi)}^{3/4} \iint_{{\Pi\setminus B_{1/4}(1,0)%\lbrace D_{1,0}\geq1/4\rbrace
				}%\Pi
			}\frac{1}{{D_{1,0}^{7/4}}%[(\br{r}-1)^{2}+\br{z}^{2}]^{7/8}
			} %|1-\phi_{1/2}| 
			(\br{r}^{1/2}|\omg|^{1/4})\nonumber\\
			&\leq\bigg\|\frac{\omg}{r}\bigg\|_{L^{\ift}(\Pi)}^{3/4} \bigg(\iint_{{\Pi\setminus B_{1/4}(1,0)%\lbrace D_{1,0}\geq1/4\rbrace
				}%\Pi
			}\frac{1}{{D_{1,0}^{7/3}}%[(\br{r}-1)^{2}+\br{z}^{2}]^{7/6}
			} %|1-\phi_{1/2}|^{4/3}
			\bigg)^{3/4} \bigg(\iint_{{\Pi\setminus B_{1/4}(1,0)%\lbrace D_{1,0}\geq1/4\rbrace
				}%\Pi
			}\br{r}^{2}|\omg|\bigg)^{1/4}\nonumber\\
			&\lesssim\bigg\|\frac{\omg}{r}\bigg\|_{L^{\ift}(\Pi)}^{3/4}\nrm{r^{2}\omg}_{L^{1}(\Pi)}^{1/4}{\sim\bigg\|\frac{\omg}{r}\bigg\|_{L^{\ift}(\bbR^{3})}^{3/4}\nrm{r\omg}_{L^{1}(\bbR^{3})}^{1/4}.}\label{eq:I1nonlocest1}
		\end{align}

		Another estimate  of $ (I_{1}'') $ follows from an integration by parts, which gives a representation 
		\begin{equation}\label{eq:I1ibp}
			\begin{split}
				(I_{1}'')=&\ \iint_{\Pi}\Big[\big[\rd_{\br{z}}\big(F^{r}(1-\phi_{1/2})\big)\big]u^{r}-\big[\rd_{\br{r}}\big(F^{r}(1-\phi_{1/2})\big)\big]u^{z}\Big]\\
				=&{\underbrace{\iint_{\Pi}(\rd_{\br{z}}F^{r})(1-\phi_{1/2})u^{r}}_{=:(I_{11}'')}-\underbrace{\iint_{\Pi}F^{r}(\rd_{\br{z}}\phi_{1/2})u^{r}}_{=:(I_{12}'')}}  {-\underbrace{\iint_{\Pi}(\rd_{\br{r}}F^{r})(1-\phi_{1/2})u^{z}}_{=:(I_{13}'')}+\underbrace{\iint_{\Pi}F^{r}(\rd_{\br{r}}\phi_{1/2})u^{z}}_{=:(I_{14}'')}.}
				%\ \iint_{\Pi}\Big[\big[(\rd_{\br{z}}F^{r})(1-\phi_{1/2})-F^{r}(\rd_{\br{z}}\phi_{1/2})\big]u^{r}-\big[(\rd_{\br{r}}F^{r})(1-\phi_{1/2})-F^{r}(\rd_{\br{r}}\phi_{1/2})\big]u^{z}\Big].
				%\bigg(\bigg)^{1/3}\bigg(\iint_{\Pi}F^{r}(1-\phi_{1/2})\omg\bigg)^{2/3}%\bigg(\iint_{\Pi}F^{r}(1-\phi_{1/2})\omg\bigg)^{1/3}\bigg(\iint_{\Pi}F^{r}(1-\phi_{1/2})\omg\bigg)^{2/3}\\
			\end{split}
		\end{equation} We compute  		
		\begin{equation}\label{eq:rdzFr}
			\begin{split}
				\rd_{\br{z}}F^{r}(1,0,\br{r},\br{z})&=\frac{1}{\pi} \frac{1}{\br{r}^{1/2}} \bigg[\calF'\bigg(\frac{{D_{1,0}^{2}}%(\br{r}-1)^{2}+\br{z}^{2}
				}{\br{r}}\bigg)+\frac{2\br{z}^{2}}{\br{r}} \calF''\bigg(\frac{{D_{1,0}^{2}}%(\br{r}-1)^{2}+\br{z}^{2}
				}{\br{r}}\bigg)\bigg],
			\end{split}
		\end{equation}
		and
		\begin{equation}\label{eq:rdrFr}
			\begin{split}
				\rd_{\br{r}}F^{r}(1,0,\br{r},\br{z})&=-\frac{1}{2\pi} \frac{\br{z}}{\br{r}^{5/2}} \bigg[\br{r} \calF'\bigg(\frac{{D_{1,0}^{2}}%(\br{r}-1)^{2}+\br{z}^{2}
				}{\br{r}}\bigg)-2[(\br{r}-1)(\br{r}+1)-\br{z}^{2}] \calF''\bigg(\frac{{D_{1,0}^{2}}%(\br{r}-1)^{2}+\br{z}^{2}
				}{\br{r}}\bigg)\bigg].
			\end{split}
		\end{equation}
		For the term \eqref{eq:rdzFr}, %on the region $ \Pi\setminus B_{1/4}((1,0)) $, 
		we  {have the estimate}%use Lemma \ref{lem:calFest} with the exponent $ -5/2 $ to get
		\begin{equation}\label{eq:rdzFrnonloc}
			\begin{split}
				|\rd_{\br{z}}F^{r}(1,0,\br{r},\br{z})|&\lesssim\frac{1}{\br{r}^{1/2}} \bigg[\frac{\br{r}^{5/2}}{{D_{1,0}^{5}}%[(\br{r}-1)^{2}+\br{z}^{2}]^{5/2}
				}+\frac{\br{z}^{2}}{\br{r}} \frac{\br{r}^{7/2}}{{D_{1,0}^{7}}%[(\br{r}-1)^{2}+\br{z}^{2}]^{7/2}
				}\bigg]\lesssim\frac{\br{r}^{2}}{{D_{1,0}^{5}}%[(\br{r}-1)^{2}+\br{z}^{2}]^{5/2}
				},
			\end{split}
		\end{equation}
		{which holds on the region $\Pi \setminus B_{1/4}(1,0)$.}
		{We use this to get}%Using this, we have
		\begin{align}
			{|(I_{11}'')|}%\bigg|\iint_{\Pi}(\rd_{\br{z}}F^{r})(1-\phi_{1/2})u^{r}\bigg|
			&\lesssim\iint_{{\Pi\setminus B_{1/4}(1,0)}}\frac{\br{r}^{3/2}}{{D_{1,0}^{5}}%[(\br{r}-1)^{2}+\br{z}^{2}]^{5/2}
			} %|1-\phi_{1/2}| 
			(\br{r}^{1/2}|u^{r}|)\lesssim\iint_{{\Pi\setminus B_{1/4}(1,0)}}\frac{1}{{D_{1,0}^{7/2}}%[(\br{r}-1)^{2}+\br{z}^{2}]^{7/4}
			} %|1-\phi_{1/2}| 
			(\br{r}^{1/2}|u^{r}|)\nonumber\\
			&\leq\bigg(\iint_{{\Pi\setminus B_{1/4}(1,0)}}\frac{1}{{D_{1,0}^{7}}%[(\br{r}-1)^{2}+\br{z}^{2}]^{7/2}
			} %|1-\phi_{1/2}|^{2}
			\bigg)^{1/2} \bigg(\iint_{{\Pi\setminus B_{1/4}(1,0)}}\br{r}|u^{r}|^{2}\bigg)^{1/2} \lesssim\nrm{r^{1/2}u}_{L^{2}(\Pi)}{\sim\nrm{u}_{L^{2}(\bbR^{3})}}.\label{eq:rdzFrphiur}
		\end{align}
		Next, we us the estimate \eqref{eq:Frnonloc} to get
		\begin{align}
			{|(I_{12}'')|}%\bigg|\iint_{\Pi}F^{r}(\rd_{\br{z}}\phi_{1/2})u^{r}\bigg|
			&\lesssim\iint_{\Pi}\frac{\br{r}^{3/2}}{{D_{1,0}^{4}}%[(\br{r}-1)^{2}+\br{z}^{2}]^{2}
			} \chi'(2{D_{1,0}}%\sqrt{(\br{r}-1)^{2}+\br{z}^{2}}
			) \frac{|\br{z}|}{{D_{1,0}}%\sqrt{(\br{r}-1)^{2}+\br{z}^{2}}
			} (\br{r}^{1/2}|u^{r}|)\lesssim\iint_{{B_{1/2}(1,0)\setminus B_{1/4}(1,0)}}\frac{1}{{D_{1,0}^{4}}%[(\br{r}-1)^{2}+\br{z}^{2}]^{2}
			}
			% \chi'(2\sqrt{(\br{r}-1)^{2}+\br{z}^{2}}) 
			(\br{r}^{1/2}|u^{r}|).\nonumber
		\end{align}
		The second inequality follows from the fact that the integrand above is supported on the region $ {D_{1,0}}%|(\br{r},\br{z})-(1,0)|
		\sim1 $. %$ \lbrace1/4\leq|(\br{r},\br{z})-(1,0)|<1/2 \rbrace $.
		Finishing this estimate, we have
		\begin{align}
			{|(I_{12}'')|}%\bigg|\iint_{\Pi}F^{r}(\rd_{\br{z}}\phi_{1/2})u^{r}\bigg|
			&\lesssim\bigg(\iint_{{B_{1/2}(1,0)\setminus B_{1/4}(1,0)}}\frac{1}{{D_{1,0}^{8}}%[(\br{r}-1)^{2}+\br{z}^{2}]^{4}
			} %|\chi'(2\sqrt{(\br{r}-1)^{2}+\br{z}^{2}})|^{2}
			\bigg)^{1/2} \bigg(\iint_{{B_{1/2}(1,0)\setminus B_{1/4}(1,0)}}\br{r}|u^{r}|^{2}\bigg)^{1/2} \lesssim \nrm{u}_{L^{2}(\bbR^{3})}.\label{eq:Frrdzphiur}
		\end{align}
		%Before estimating 
		{To estimate} the term \eqref{eq:rdrFr}, 
		{we use the  {inequality}%relation
			\begin{comment}\label{eq:r2z2}
				%\Pi\setminus B_{1/4}(1,0)\subset \lbrace (\br{r},\br{z})\in \Pi : |(\br{r}-1)(\br{r}+1)-\br{z}^{2}|^{1/2}\leq 10D_{1,0}\rbrace.
			\end{comment}
			%This allows us to use the inequality
			$$ |(\br{r}-1)(\br{r}+1)-\br{z}^{2}|^{1/2}\lesssim D_{1,0}, $$
			{which holds }on the region $\Pi\setminus B_{1/4}(1,0)$.}
		We use this to estimate the term \eqref{eq:rdrFr} as
		\begin{align}
			|\rd_{\br{r}}F^{r}(1,0,\br{r},\br{z})|&\lesssim\frac{|\br{z}|}{\br{r}^{5/2}} \bigg[\br{r} \frac{\br{r}^{5/2}}{{D_{1,0}^{5}}%[(\br{r}-1)^{2}+\br{z}^{2}]^{5/2}
			}+|(\br{r}-1)(\br{r}+1)-\br{z}^{2}| \frac{\br{r}^{7/2}}{{D_{1,0}^{7}}%[(\br{r}-1)^{2}+\br{z}^{2}]^{7/2}
			}\bigg]%\\%\lesssim\frac{\br{r}}{[(\br{r}-1)^{2}+\br{z}^{2}]^{2}}.
			%&
			\lesssim\frac{\br{r}}{{D_{1,0}^{4}}%[(\br{r}-1)^{2}+\br{z}^{2}]^{2}
			},\label{eq:rdrFrnonloc}
		\end{align}
		{where it holds on the region $\Pi \setminus B_{1/4}(1,0)$ as well.} 
		Using this, we  {get}%have
		\begin{align}
			{|(I_{13}'')|}%\bigg|\iint_{\Pi}(\rd_{\br{r}}F^{r})(1-\phi_{1/2})u^{r}\bigg|
			&\lesssim\iint_{{\Pi\setminus B_{1/4}(1,0)}}\frac{\br{r}^{1/2}}{{D_{1,0}^{4}}%[(\br{r}-1)^{2}+\br{z}^{2}]^{2}
			} %|1-\phi_{1/2}|
			(\br{r}^{1/2}|u^{{z}}|)\lesssim\iint_{{\Pi\setminus B_{1/4}(1,0)}}\frac{1}{{D_{1,0}^{7/2}}%[(\br{r}-1)^{2}+\br{z}^{2}]^{7/4}
			} %|1-\phi_{1/2}|
			(\br{r}^{1/2}|u^{{z}}|) \lesssim {\nrm{u}_{L^{2}(\bbR^{3})}}.\label{eq:rdrFrphiur}
		\end{align}
		Lastly, we obtain
		\begin{align}
			{|(I_{14}'')|}%\bigg|\iint_{\Pi}F^{r}(\rd_{\br{r}}\phi_{1/2})u^{r}\bigg|
			&\lesssim\iint_{\Pi}\frac{\br{r}^{3/2}}{{D_{1,0}^{4}}%[(\br{r}-1)^{2}+\br{z}^{2}]^{2}
			} \chi'(2{D_{1,0}}%\sqrt{(\br{r}-1)^{2}+\br{z}^{2}}
			) \frac{|\br{r}-1|}{{D_{1,0}}%\sqrt{(\br{r}-1)^{2}+\br{z}^{2}}
			} (\br{r}^{1/2}|u^{{z}}|)\lesssim\iint_{{B_{1/2}(1,0)\setminus B_{1/4}(1,0)}}\frac{\br{r}^{1/2}|u^{{z}}|}{{D_{1,0}^{4}}%[(\br{r}-1)^{2}+\br{z}^{2}]^{2}
			} %\chi'(2\sqrt{(\br{r}-1)^{2}+\br{z}^{2}}) 
			\lesssim {\nrm{u}_{L^{2}(\bbR^{3})}}.\label{eq:Frrdrphiur}
		\end{align}
		Now we may collect the estimates  {above}  %\eqref{eq:I1nonlocest1}, \eqref{eq:rdzFrphiur}, \eqref{eq:Frrdzphiur}, \eqref{eq:rdrFrphiur}, and \eqref{eq:Frrdrphiur} 
		to conclude
		\begin{align}
			|(I_{1}'')|&=|(I_{1}'')|^{1/3}|(I_{1}'')|^{2/3}\lesssim{
				\nrm{u}_{L^{2}(\bbR^{3})%\nrm{r^{1/2}u}_{L^{2}(\Pi)
					}^{1/3}
					\bigg(
					\bigg\|\frac{\omg}{r}\bigg\|_{L^{\ift}(\bbR^{3})}^{3/4}\nrm{r\omg}_{L^{1}(\bbR^{3})}^{1/4}%\bigg\|\frac{\omg}{r}\bigg\|_{L^{\ift}(\Pi)}^{3/4}\nrm{r^{2}\omg}_{L^{1}(\Pi)}^{1/4}
					\bigg)^{2/3}}\nonumber\\
				&=\nrm{u}_{L^{2}(\bbR^{3})}%\nrm{r^{1/2}u}_{L^{2}(\Pi)}
				^{1/3}
				%\bigg\|\frac{\omg}{r}\bigg\|_{L^{\ift}(\Pi)}
				\bigg\|\frac{\omg}{r}\bigg\|_{L^{\ift}(\bbR^{3})}^{1/2}
				\nrm{r\omg}_{L^{1}(\bbR^{3})}%\nrm{r^{2}\omg}_{L^{1}(\Pi)}
				^{1/6}.
				\label{eq:I1dp}
			\end{align}

			\medskip
			
			\noindent \textbf{Estimate of $(I_{2}'') $:  {splitting further.}} To estimate $ (I_{2}'') $,  {we define the ratio}%we let $ k>0 $
			\begin{equation}\label{eq:ratiok}
				k={\nrm{u}_{L^{2}(\bbR^{3})}
					^{{2/3}}}%\nrm{r^{1/2}u}_{L^{2}(\Pi)}
				{\bigg\|\frac{\omg}{r}\bigg\|_{L^{\ift}(\bbR^{3})}
					^{{-1/2}}}%\bigg\|\frac{\omg}{r}\bigg\|_{L^{\ift}(\Pi)}
				{\nrm{r\omg}_{L^{1}(\bbR^{3})}
					^{{-1/6}},}%\nrm{r^{2}\omg}_{L^{1}(\Pi)}
			\end{equation} 
			and split $ (I_{2}'') $ into
			\begin{equation*}
				\begin{split}
					%(I_{2}'')=
					\underbrace{\iint_{\Pi}F^{r}(1-\phi_{k})\phi_{1/2}\omg}_{=:(I_{21}'')}+\underbrace{\iint_{\Pi}F^{r}\phi_{k}\phi_{1/2}\omg}_{=:(I_{22}'')}.
				\end{split}
			\end{equation*}
			
			\medskip
			
			\noindent \textbf{Estimate of $(I_{22}'')$.}  {We define $\bar{k}:=\min\lbrace k,1\rbrace$ and estimate $(I_{22}'')$ as}
			%We consider cases $ k>1 $ and $ k\leq1 $.  When $ k>1 $, we have $ (I_{21}'')=0$ 
			%and
			\begin{align}
				|(I_{22}'')|&{\leq}%=
				%\bigg|
				\iint_{{B_{\bar{k}}(1,0)}}|F^{r}%\phi_{1/2}
				\omg|%\bigg|
				\lesssim\iint_{{B_{\bar{k}}(1,0)}}\frac{\br{r}^{1/2}}{{D_{1,0}}%\sqrt{(\br{r}-1)^{2}+\br{z}^{2}}
				} |%\phi_{1/2}
				\omg|^{3/4} |%\phi_{1/2}
				\omg|^{1/4} 
				\frac{\br{r}^{3/4}}{\br{r}^{3/4}}\\
				&\lesssim\bigg\|\frac{\omg}{r}\bigg\|_{L^{\ift}({B_{\bar{k}}(1,0)})}^{3/4}%\nrm{%\phi_{1/2}
					%\omg}_{L^{\ift}({B_{1/2}(1,0)})}^{3/4} 
				\bigg(\iint_{{B_{\bar{k}}(1,0)}}\frac{1}{{D_{1,0}^{4/3}}%[(\br{r}-1)^{2}+\br{z}^{2}]^{2/3}
				} %|\phi_{1/2}|^{1/6}
				\bigg)^{3/4} \bigg(\iint_{{B_{\bar{k}}(1,0)}}%\sqrt{\phi_{1/2}}
				{\br{r}^{2}}|\omg|\bigg)^{1/4}\\
				&%\lesssim\nrm{%\phi_{1/2}
					%\omg}_{L^{\ift}({B_{1/2}(1,0)})}^{3/4}\nrm{%\sqrt{\phi_{1/2}}
					%\omg}_{L^{1}({B_{1/2}(1,0)})}^{1/4}
				\lesssim \sqrt{\bar{k}} \bigg\|\frac{\omg}{r}\bigg\|_{L^{\ift}({B_{\bar{k}}(1,0)})}^{3/4}\nrm{r^{2}\omg}_{L^{1}({B_{\bar{k}}(1,0)})}^{1/4}\lesssim\sqrt{\bar{k}} \bigg\|\frac{\omg}{r}\bigg\|_{L^{\ift}(\bbR^{3})}^{3/4}\nrm{r\omg}_{L^{1}(\bbR^{3})}^{1/4}\nonumber\\
				&\leq\nrm{u}_{L^{2}(\bbR^{3})}^{1/3}\bigg\|\frac{\omg}{r}\bigg\|_{L^{\ift}(\bbR^{3})}^{1/2}\nrm{r\omg}_{L^{1}(\bbR^{3})}^{1/6}.\label{eq:I22dp}
			\end{align}
			%The {fourth} inequality follows because we assumed $ k>1 $ and $ r $ satisfies $ r\sim1 $ on the disc $ B_{1/2}(1,0) $.
			\medskip
			
			\noindent \textbf{Estimate of $(I_{21}'')$.}  {Let us assume $k<1$ because $(I_{21}'')$ just vanishes otherwise.} %For the case $ k\leq1 $, 
			We use integration by parts to get
			\begin{equation}\label{eq:I21ibp}
				\begin{split}
					(I_{21}'')=&\ \iint_{\Pi}\Big[\ \big[\rd_{\br{z}}\big(F^{r}(1-\phi_{k})\phi_{1/2}\big)\big]u^{r}-\big[\rd_{\br{r}}\big(F^{r}(1-\phi_{k})\phi_{1/2}\big)\big]u^{z}\Big]\\
					=&\ {\underbrace{\iint_{\Pi}(\rd_{\br{z}}F^{r})(1-\phi_{k})\phi_{1/2}u^{r}}_{=:(I_{211}'')}-\underbrace{\iint_{\Pi}F^{r}(\rd_{\br{z}}\phi_{k})\phi_{1/2}u^{r}}_{=:(I_{212}'')}+\underbrace{\iint_{\Pi}F^{r}(1-\phi_{k})(\rd_{\br{z}}\phi_{1/2})u^{r}}_{=:(I_{213}'')}}\\
					&{-\underbrace{\iint_{\Pi}(\rd_{\br{r}}F^{r})(1-\phi_{k})\phi_{1/2}u^{z}}_{=:(I_{214}'')}+\underbrace{\iint_{\Pi}F^{r}(\rd_{\br{r}}\phi_{k})\phi_{1/2}u^{z}}_{=:(I_{215}'')}-\underbrace{\iint_{\Pi}F^{r}(1-\phi_{k})(\rd_{\br{r}}\phi_{1/2})u^{z}}_{=:(I_{216}'')}.}
					%\iint_{\Pi}\Big[&\ \big[(\rd_{\br{z}}F^{r})(1-\phi_{k})\phi_{1/2}-F^{r}(\rd_{\br{z}}\phi_{k})\phi_{1/2}+F^{r}(1-\phi_{k})(\rd_{\br{z}}\phi_{1/2})\big]u^{r}\\
					%&-\big[(\rd_{\br{r}}F^{r})(1-\phi_{k})\phi_{1/2}-F^{r}(\rd_{\br{r}}\phi_{k})\phi_{1/2}+F^{r}(1-\phi_{k})(\rd_{\br{r}}\phi_{1/2})\big]u^{z}\Big].
					%\bigg(\bigg)^{1/3}\bigg(\iint_{\Pi}F^{r}(1-\phi_{1/2})\omg\bigg)^{2/3}%\bigg(\iint_{\Pi}F^{r}(1-\phi_{1/2})\omg\bigg)^{1/3}\bigg(\iint_{\Pi}F^{r}(1-\phi_{1/2})\omg\bigg)^{2/3}\\
				\end{split}
			\end{equation}
			Now let us estimate each term. First, we get
			\begin{align}
				{|(I_{211}'')|}%\bigg|\iint_{\Pi}(\rd_{\br{z}}F^{r})(1-\phi_{k})\phi_{1/2}u^{r}\bigg|
				&\lesssim\iint_{{B_{1/2}(1,0)\setminus B_{k/2}(1,0)}}\frac{\br{r}^{1/2}}{{D_{1,0}^{2}}%(\br{r}-1)^{2}+\br{z}^{2}
				} %|1-\phi_{k}| 
				|%\phi_{1/2}
				u^{r}|%\lesssim\iint_{{B_{1/2}(1,0)\setminus B_{k/2}(1,0)}}\frac{1}{{D_{1,0}^{2}}%(\br{r}-1)^{2}+\br{z}^{2}
					%} %|1-\phi_{k}| 
				%|%\phi_{1/2}
				%u^{r}|
				\nonumber\\
				&\leq\bigg(\iint_{{B_{1/2}(1,0)\setminus B_{k/2}(1,0)}}\frac{1}{{D_{1,0}^{4}}%[(\br{r}-1)^{2}+\br{z}^{2}]^{2}
				} %|1-\phi_{k}|^{2}
				\bigg)^{1/2} \bigg(\iint_{{B_{1/2}(1,0)\setminus B_{k/2}(1,0)}}{\br{r}}|%\phi_{1/2}
				u^{r}|^{2}\bigg)^{1/2} 
				\lesssim  \frac{1}{k} \nrm{u}_{L^{2}(\bbR^{3})}.\label{eq:rdzFrphikphi1/2ur}
			\end{align}
			Next, we have
			\begin{align}
				{|(I_{212}'')|}%\bigg|\iint_{\Pi}F^{r}(\rd_{\br{z}}\phi_{k})\phi_{1/2}u^{r}\bigg|
				&\lesssim\iint_{{B_{1/2}(1,0)}}\frac{\br{r}^{1/2}}{{D_{1,0}}%\sqrt{(\br{r}-1)^{2}+\br{z}^{2}}
				} \frac{1}{k} \chi'\bigg(\frac{{D_{1,0}}%\sqrt{(\br{r}-1)^{2}+\br{z}^{2}}
				}{k}\bigg) \frac{|\br{z}|}{{D_{1,0}}%\sqrt{(\br{r}-1)^{2}+\br{z}^{2}}
				} |%\phi_{1/2}
				u^{r}|\lesssim\iint_{{B_{1/2}(1,0)\setminus B_{k/2}(1,0)}}\frac{1}{{D_{1,0}^{2}}%(\br{r}-1)^{2}+\br{z}^{2}
				} %\chi'\bigg(\frac{\sqrt{(\br{r}-1)^{2}+\br{z}^{2}}}{k}\bigg) 
				{\br{r}^{1/2}}|%\phi_{1/2}
				u^{r}|.\nonumber\\
				&\leq\bigg(\iint_{{B_{1/2}(1,0)\setminus B_{k/2}(1,0)}}\frac{1}{{D_{1,0}^{4}}%[(\br{r}-1)^{2}+\br{z}^{2}]^{2}
				} %\bigg|\chi'\bigg(\frac{\sqrt{(\br{r}-1)^{2}+\br{z}^{2}}}{k}\bigg)\bigg|^{2}
				\bigg)^{1/2} \bigg(\iint_{{B_{1/2}(1,0)\setminus B_{k/2}(1,0)}}{\br{r}}|%\phi_{1/2}
				u^{r}|^{2}\bigg)^{1/2} 
				\lesssim  { \frac{1}{k} \nrm{u}_{L^{2}(\bbR^{3})}}.\label{eq:Frrdzphikphi1/2ur}
			\end{align}
			In the second inequality, we used the relation ${D_{1,0}}%\sqrt{(\br{r}-1)^{2}+\br{z}^{2}}
			\sim k$, which holds because of the term $\rd_{\br{z}}\phi_{k}$. 
			Also, {we use the relation $D_{1,0}\sim1$ coming from the term $\rd_{\br{z}}\phi_{1/2}$ to get} %we get
			\begin{align}
				{|(I_{213}'')|}%\bigg|\iint_{\Pi}F^{r}(1-\phi_{k})(\rd_{\br{z}}\phi_{1/2})u^{r}\bigg|
				&\lesssim\iint_{{\Pi\setminus B_{k/2}(1,0)}}\frac{\br{r}^{1/2}}{{D_{1,0}}%(\br{r}-1)^{2}+\br{z}^{2}
				} %|1-\phi_{k}| 
				\chi'(2{D_{1,0}}%\sqrt{(\br{r}-1)^{2}+\br{z}^{2}}
				) \frac{|\br{z}|}{{D_{1,0}}%\sqrt{(\br{r}-1)^{2}+\br{z}^{2}}
				} |u^{r}|\lesssim\iint_{{B_{1/2}(1,0)\setminus B_{k/2}(1,0)}}\frac{1}{{D_{1,0}^{2}}%(\br{r}-1)^{2}+\br{z}^{2}
				} %\chi'\bigg(\frac{\sqrt{(\br{r}-1)^{2}+\br{z}^{2}}}{k}\bigg) 
				{\br{r}^{1/2}}|%\phi_{1/2}
				u^{r}| 
				\lesssim {\frac{1}{k} \nrm{u}_{L^{2}(\bbR^{3})}}.%{\frac{1}{k^{2}}} \nrm{\phi u^{r}}_{L^{2}(\Pi)}\lesssim {\frac{1}{k^{2}}} \nrm{r^{1/2}u}_{L^{2}(\Pi)}.
				\label{eq:Frphikrdzphi1/2ur}
			\end{align}
			\begin{comment}
				Now observe that by choosing $ c_{0}=50/k^{2} $, we get
				$$ R_{2}(c_{0})\subset R_{1}(c_{0})\subset B_{k}((1,0)), $$
				which implies
				$$ \Pi\setminus B_{k}((1,0))\subset Q_{1}(c_{0})\cap Q_{2}(c_{0}). $$
			\end{comment}
			%Using this, we have
			Now  we can use the inequality $ |(\br{r}-1)(\br{r}+1)-\br{z}^{2}|\lesssim{D_{1,0}}$  on the disc $B_{1/2}(1,0)$ to get
			\begin{align}
				|\rd_{\br{r}}F^{r}(1,0,\br{r},\br{z})|&\lesssim\frac{|\br{z}|}{\br{r}^{5/2}} \bigg[\br{r} \frac{\br{r}}{{D_{1,0}^{2}}%(\br{r}-1)^{2}+\br{z}^{2}
				}+|(\br{r}-1)(\br{r}+1)-\br{z}^{2}| \frac{\br{r}^{2}}{{D_{1,0}^{4}}%[(\br{r}-1)^{2}+\br{z}^{2}]^{2}
				}\bigg] \lesssim\frac{1}{{D_{1,0}}%\sqrt{(\br{r}-1)^{2}+\br{z}^{2}}
				}+\frac{1}{{D_{1,0}^{2}}%(\br{r}-1)^{2}+\br{z}^{2}
				}\lesssim\frac{1}{{D_{1,0}^{2}}%(\br{r}-1)^{2}+\br{z}^{2}
				}.\label{eq:rdrFrloc}
			\end{align}
			Using this, we have
			\begin{align}
				{|(I_{214}'')|}%\bigg|\iint_{\Pi}(\rd_{\br{r}}F^{r})(1-\phi_{k})\phi_{1/2}u^{r}\bigg|%&\lesssim\frac{1}{k^{2}} \iint_{\Pi}\frac{1}{r^{1/2}\sqrt{(\br{r}-1)^{2}+\br{z}^{2}}} |1-\phi_{k}| |\phi_{1/2}u^{r}|\nonumber\\
				&\lesssim\iint_{{B_{1/2}(1,0)\setminus B_{k/2}(1,0)}}\frac{1}{{D_{1,0}^{2}}%(\br{r}-1)^{2}+\br{z}^{2}
				} %|1-\phi_{k}| 
				|%\phi_{1/2}
				u^{{z}}|\lesssim\iint_{{B_{1/2}(1,0)\setminus B_{k/2}(1,0)}}\frac{1}{{D_{1,0}^{2}}%(\br{r}-1)^{2}+\br{z}^{2}
				} %|1-\phi_{k}| 
				{\br{r}^{1/2}}|%\phi_{1/2}
				u^{{z}}|  {\lesssim\frac{1}{k} \nrm{u}_{L^{2}(\bbR^{3})}}\label{eq:rdrFrphikphi1/2ur}
			\end{align}
			and similarly, we get
			\begin{equation}\label{eq:Frrdrphikphi1/2ur}
				{|(I_{215}'')|+|(I_{216}'')|}%\bigg|\iint_{\Pi}F^{r}(\rd_{\br{z}}\phi_{k})\phi_{1/2}u^{r}\bigg|
				\lesssim%\frac{1}{k} 
				{\frac{1}{k} \nrm{u}_{L^{2}(\bbR^{3})}}.%\nrm{r^{1/2}u}_{L^{2}(\Pi)},
			\end{equation}
			Thus, gathering all the estimates above, we obtain
			\begin{equation}\label{eq:I21dp}
				\begin{split}
					|(I_{21}'')|& {\lesssim%\sqrt{k} {\bigg\|\frac{\omg}{r}\bigg\|_{L^{\ift}(\bbR^{3})}^{3/4}\nrm{r\omg}_{L^{1}(\bbR^{3})}^{1/4}}+\frac{1}{k} 
						%{\nrm{u}_{L^{2}(\bbR^{3})}}\\
						\frac{1}{k} \nrm{u}_{L^{2}(\bbR^{3})}\leq\nrm{u}_{L^{2}(\bbR^{3})}^{1/3}\bigg\|\frac{\omg}{r}\bigg\|_{L^{\ift}(\bbR^{3})}^{1/2}\nrm{r\omg}_{L^{1}(\bbR^{3})}^{1/6}.}
				\end{split}
			\end{equation}
			{We used the definition of $k$ in the last inequality.}
			
			\medskip
			
			\noindent {\textbf{Conclusion.} Summing up \eqref{eq:I22dp} and \eqref{eq:I21dp}, we have}
			\begin{equation}\label{eq:I2dp}
				{|(I_{2}'')|\lesssim\nrm{u}_{L^{2}(\bbR^{3})}^{1/3}\bigg\|\frac{\omg}{r}\bigg\|_{L^{\ift}(\bbR^{3})}^{1/2}\nrm{r\omg}_{L^{1}(\bbR^{3})}^{1/6}.}
			\end{equation}
			{Hence, adding \eqref{eq:I1dp} and \eqref{eq:I2dp}, the proof is done.} 
		\end{proof} 
		\begin{comment}
			We borrow the same $R(t)$ from the proof of Theorem \ref{thm:growth-optimal}. Then for any $t\geq0$, we use the estimate \eqref{eq:newest} to get
			\begin{equation*}
				\begin{split}
					\bigg|\frac{d}{dt}R(t)\bigg|&\leq \nrm{u^{r}(t,\cdot)}_{L^{\ift}(\bbR^{3})} \leq C\nrm{u(t,\cdot)}_{L^{2}(\bbR^{3})}^{1/3}\bigg\|\frac{\omg^{\tht}(t,\cdot)}{r}\bigg\|_{L^{\ift}(\bbR^{3})}^{1/2}\nrm{r\omg^{\tht}(t,\cdot)}_{L^{1}(\bbR^{3})}^{1/6}\\
					&\leq C\nrm{u(t,\cdot)}_{L^{2}(\bbR^{3})}^{1/3}\bigg\|\frac{\omg^{\tht}(t,\cdot)}{r}\bigg\|_{L^{\ift}(\bbR^{3})}^{1/2}R^{1/3}(t)\bigg\|\frac{\omg^{\tht}(t,\cdot)}{r}\bigg\|_{L^{1}(\bbR^{3})}^{1/6}\\
					&\leq Cc(\omg_{0})R^{1/3}(t).
				\end{split}
			\end{equation*}
			Here, we used conservations of the kinetic energy $\nrm{u}_{L^{2}}$ and the $L^{p}$-norms of $r^{-1}\omg$ with $p\in [1,\ift]$. From this, we get the bound
			\begin{equation*}
				R(t)\lesssim_{\omg_{0}} (1+t)^{3/2},\quad t\geq0,
			\end{equation*}
			and this gives us the desired estimate
			\begin{equation*}
				\nrm{\omg(t,\cdot)}_{L^{\ift}(\bbR^{3})}\leq \bigg\|\frac{\omg(t,\cdot)}{r}\bigg\|_{L^{\ift}(\bbR^{3})}R(t)= \bigg\|\frac{\omg_{0}}{r}\bigg\|_{L^{\ift}(\bbR^{3})}R(t)\lesssim_{\omg_{0}} (1+t)^{3/2}.
			\end{equation*}
		\end{comment}
		
		{Now we prove Theorem \ref{thm:t3/2} using Proposition \ref{prop:estimate}. \begin{proof}[Proof of Theorem \ref{thm:t3/2}]
				We let $t\geq0$, and we define the \textit{length} function $L(t)$ as
				\begin{equation}\label{eq:lengthftn}
					L(t):=1+\int_{0}^{t}\nrm{u^{r}(s,\cdot)}_{L^{\ift}(\bbR^{3})}ds.%,\quad t\geq0.
				\end{equation}
				
				\medskip
				
				\noindent \textbf{Claim}. We have the estimates
				\begin{align}
					%\text{\textbf{Claim 1.}}&\quad 
					\nrm{r\omg(t,\cdot)}_{L^{1}(\bbR^{3})}&\leq \bigg(\bigg\|\frac{\omg_{0}}{r}\bigg\|_{L^{1}(\bbR^{3})}+\nrm{r\omg_{0}}_{L^{1}(\bbR^{3})}\bigg)L(t)^{2},\label{eq:claim1}\\
					%\text{\textbf{Claim 2.}}&\quad 
					\nrm{\omg(t,\cdot)}_{L^{\ift}(\bbR^{3})}&\leq \bigg(\bigg\|\frac{\omg_{0}}{r}\bigg\|_{L^{\ift}(\bbR^{3})}+\nrm{\omg_{0}}_{L^{\ift}(\bbR^{3})}\bigg)L(t).\label{eq:claim2}
				\end{align}
				Once this \textbf{Claim} is shown, then  \eqref{eq:claim1}  and conservation of $\nrm{u}_{L^{2}}$ and the $L^{p}$-norms of $r^{-1}\omg$ with $p\in [1,\ift]$ imply that we have
				\begin{equation*}
					\begin{split}
						\frac{d}{dt}L(t)&=\nrm{u^{r}(t,\cdot)}_{L^{\ift}(\bbR^{3})}\\
						&\lesssim\nrm{u(t,\cdot)}_{L^{2}(\bbR^{3})}^{1/3}\bigg\|\frac{\omg(t,\cdot)}{r}\bigg\|_{L^{\ift}(\bbR^{3})}^{1/2}\nrm{r\omg(t,\cdot)}_{L^{1}(\bbR^{3})}^{1/6}\\
						&\leq\nrm{u_{0}}_{L^{2}(\bbR^{3})}^{1/3}\bigg\|\frac{\omg_{0}}{r}\bigg\|_{L^{\ift}(\bbR^{3})}^{1/2}\bigg(\bigg\|\frac{\omg_{0}}{r}\bigg\|_{L^{1}(\bbR^{3})}+\nrm{r\omg_{0}}_{L^{1}(\bbR^{3})}\bigg)^{1/6}L(t)^{1/3}%L^{1/3}(t)\bigg\|\frac{\omg_{0}}{r}\bigg\|_{L^{1}(\bbR^{3})}^{1/6}
						\lesssim_{\omg_{0}}L(t)^{1/3}.
					\end{split}
				\end{equation*}
				From this, we have
				\begin{equation*}
					L(t)\lesssim_{\omg_{0}}(1+t)^{3/2},
				\end{equation*}
				where this and the second estimate \eqref{eq:claim2} from our claim give us the desired estimate
				\begin{equation}
					\nrm{\omg(t,\cdot)}_{L^{\ift}(\bbR^{3})}\lesssim_{\omg_{0}}(1+t)^{3/2}.
				\end{equation}
				
				It only remains to prove the \textbf{Claim}. To do this, we use the flow map $\Phi_{t}$ which solves the ODE
				\begin{equation*}
					\frac{d}{dt}\Phi_{t}(a)=u(t,\Phi_{t}(a)),\quad \Phi_{0}(a)=a.
				\end{equation*}
				We let $x\in \bbR^{3}$ and use notations $r_{x}:=(x_{1}^{2}+x_{2}^{2})^{1/2}$ and $z_{x}:=x_{3}$. %denote $r_{x} $ and $ z_{x}$ as the radial and vertical component of $x$, respectively(i.e., ). 
				In addition, we denote $\Phi_{t}^{r}$ and $\Phi_{t}^{z}$ as the radial and vertical component of $\Phi_{t}$, respectively. Also, we take $y=(\Phi_{t})^{-1}(x)$, where $(\Phi_{t})^{-1}$ is the inverse map of $\Phi_{t}$. Then note that we have the following equality:
				\begin{equation*}
					x=\Phi_{t}(y)=\Phi_{0}(y)+\int_{0}^{t}u(s,\Phi_{s}(y))ds=y+\int_{0}^{t}u(s,\Phi_{s}(y))ds,
				\end{equation*}
				where its radial component would be written as
				\begin{equation}\label{eq:flowmaprad}
					r_{x}=\Phi_{t}^{r}(y)=r_{y}+\int_{0}^{t}u^{r}(s,\Phi_{s}(y))ds.
				\end{equation}
				Additionally, we can use the fact that the quantity $r^{-1}\omg$ is conserved in time along the flow map $\Phi_{t}$:
				\begin{equation}\label{eq:flowmapcons}
					\frac{\omg(t,x)}{r_{x}}=\frac{\omg(t,\Phi_{t}(y))}{\Phi_{t}^{r}(y)}=\frac{\omg_{0}(y)}{r_{y}}.
				\end{equation}
				To prove the first estimate \eqref{eq:claim1}, we define the set $\Omg(t):=\lbrace x\in \bbR^{3} : r_{x}\leq L(t)\rbrace$ and split the integral range of the $L^{1}$-norm of $r\omg(t,\cdot)$ into
				\begin{equation*}
					\nrm{r\omg(t)}_{L^{1}(\bbR^{3})}=\underbrace{\int_{\Omg(t)}r_{x}|\omg(t,x)|dx}_{=:(A)}+\underbrace{\int_{\Omg(t)^{c}}r_{x}|\omg(t,x)|dx}_{=:(B)}.
				\end{equation*}
				The term $(A)$ is estimated as
				\begin{equation}\label{eq:estA}
					\begin{split}
						(A)&=\int_{\Omg(t)}\frac{|\omg(t,x)|}{r_{x}}r_{x}^{2}dx\leq \bigg(\int_{\Omg(t)}\frac{|\omg(t,x)|}{r_{x}}dx\bigg)L(t)^{2}\\
						&\leq \bigg\|\frac{\omg_{0}}{r}\bigg\|_{L^{1}(\bbR^{3})}L(t)^{2}.
					\end{split}
				\end{equation}
				To estimate the term $(B)$, we observe that we have the relation
				\begin{equation}\label{eq:Omgtcinv}
					\begin{split}
						(\Phi_{t})^{-1}(\Omg(t)^{c})&=\lbrace (\Phi_{t})^{-1}(x)\in \bbR^{3} : x\in \Omg(t)^{c}\rbrace=\lbrace y\in \bbR^{3} : r_{x}>L(t)\rbrace\\
						&\subset \lbrace y\in \bbR^{3} : r_{y}>1\rbrace.
					\end{split}
				\end{equation}
				Also, in the region $\lbrace y\in \bbR^{3} : r_{y}>1\rbrace$, we have the inequality
				\begin{equation}\label{eq:rxryratio}
					\frac{\Phi_{t}^{r}(y)}{r_{y}}=1+\frac{1}{r_{y}}\int_{0}^{t}u^{r}(s,\Phi_{s}(y))ds\leq 1+\int_{0}^{t}u^{r}(s,\Phi_{s}(y))ds\leq 
					L(t).
				\end{equation}
				Using these two observations, we can estimate the term $(B)$:
				\begin{equation}\label{eq:estB}
					\begin{split}
						(B)&=\int_{\Omg(t)^{c}}\frac{|\omg(t,x)|}{r_{x}}r_{x}^{2}dx=\int_{(\Phi_{t})^{-1}(\Omg(t)^{c})}\frac{|\omg_{0}(y)|}{r_{y}}[\Phi_{t}^{r}(y)]^{2}dy\\
						&\leq \int_{\lbrace r_{y}>1\rbrace}\frac{|\omg_{0}(y)|}{r_{y}}r_{y}^{2}\bigg[\frac{\Phi_{t}^{r}(y)}{r_{y}}\bigg]^{2}dy\leq \nrm{r\omg_{0}}_{L^{1}(\bbR^{3})}L(t)^{2}.
					\end{split}
				\end{equation}
				Gathering estimates \eqref{eq:estA} and \eqref{eq:estB}, we have shown the first estimate \eqref{eq:claim1}. Now to prove the second estimate \eqref{eq:claim2}, we consider the case $x\in \Omg(t)$ and $x\in \Omg(t)^{c}$. For the first case $x\in \Omg(t)$, we use the conservation \eqref{eq:flowmapcons} to get
				\begin{equation}\label{eq:case1}
					|\omg(t,x)|=\frac{|\omg(t,x)|}{r_{x}}r_{x}\leq\frac{|\omg_{0}(y)|}{r_{y}}L(t)\leq\bigg\|\frac{\omg_{0}}{r}\bigg\|_{L^{\ift}(\bbR^{3})}L(t).
				\end{equation}
				For the other case $x\in \Omg(t)^{c}$, we use the conservation \eqref{eq:flowmapcons} and the inequality \eqref{eq:rxryratio} to get
				\begin{equation}\label{eq:case2}
					\begin{split}
						|\omg(t,x)|&=\frac{|\omg(t,x)|}{r_{x}}r_{x}=|\omg_{0}(y)|\frac{\Phi_{t}^{r}(y)}{r_{y}}\leq \nrm{\omg_{0}}_{L^{\ift}(\bbR^{3})}L(t).
					\end{split}
				\end{equation}
				Summing up estimates \eqref{eq:case1} and \eqref{eq:case2}, %and taking supremum with respect to $x\in \bbR^{3}$, 
				we obtain the second estimate \eqref{eq:claim2}. This finishes the proof.
			\end{proof}
		}

		\medskip

		\subsection*{Data availability statement} Data sharing is not applicable to this article as no new data were created or analyzed in this study.	
		
		\bibliographystyle{plain}

	\end{document}